\documentclass[a4paper,reqno]{amsart}
\usepackage{enumerate}
\usepackage{amssymb} 
\usepackage[colorlinks=true,urlcolor=red,citecolor=black,linkcolor=blue,pdfstartview=FitH]{hyperref}
\usepackage{indentfirst}
\usepackage[all]{xy}

\newtheorem{bigthm}{Theorem}

\newtheorem{thm}{Theorem}[section]
\newtheorem{lemma}[thm]{Lemma}
\newtheorem{prop}[thm]{Proposition}
\newtheorem{cor}[thm]{Corollary}
\newtheorem{example}[thm]{Example}
\newtheorem{rmk}[thm]{Remark}     %wh 23.5.

\numberwithin{equation}{section}

\newcommand{\ignore}[1]{} %saves typing millions of '%'.

\emergencystretch=\maxdimen
\begin{document}

%\linenumbers %to make \usepackage{lineno} works!

\title[Subgroups of pro-$p$ groups acting on pro-$p$ trees]
{Splitting theorems for pro-$p$ groups acting on pro-$p$ trees and
2-generated pro-$p$ subgroups of free pro-$p$ products with procyclic amalgamations}
\author{Wolfgang Herfort}
\address{University of Technology at Vienna, Austria}
\email{w.herfort@tuwien.ac.at}
\author{Pavel Zalesskii}
\address{Department of Mathematics, University of Brasilia, Brazil}
\email{pz@mat.unb.br}
\author{Theo Zapata}
\address{Department of Mathematics, University of Brasilia, Brazil}
\email{zapata@mat.unb.br}

%\dedicatory{}

\date{\today}
\address{\hfill\upshape\today.}

\maketitle

\vspace{-7pt}

\begin{quote}
\footnotesize
\textsc{Abstract.}
Let $G$ be an infinite finitely generated
pro-$p$ group acting on a pro-$p$ tree 
such that the restriction of the action to some open subgroup is free.
We prove that $G$ splits over an edge stabilizer
either as an amalgamated free pro-$p$ product or
as a pro-$p$ $\operatorname{HNN}$-extension.
Using this result we prove under a certain condition that
free pro-$p$ products with procyclic amalgamation 
inherit from its amalgamated free factors
the property of each $2$-generated pro-$p$ subgroup being free pro-$p$.
This generalizes known pro-$p$ results, as well as some pro-$p$
analogues of classical results in abstract combinatorial group theory.
\end{quote}

\vspace{21pt}

\begingroup
%\nolinenumbers
%\tableofcontents
\endgroup

%%%%%%%%%%%%%%%%%%%%%%%%%%%%%%
\phantomsection
\section{Introduction}
\label{s:intro}
%%%%%%%%%%%%%%%%%%%%%%%%%%%%%%

The main theorem of the Bass-Serre theory of groups acting on
trees states that a group $G$ acting on a tree $T$ is the
fundamental group of a connected graph of groups
whose vertex and edge groups are the stabilizers of 
certain vertices and edges of $T$.
This tells that $G$ can be obtained by successively forming
amalgamated free products and $\operatorname{HNN}$-extensions.
The pro-$p$ version of this theorem does not hold in general
(\textit{cf.} \hyperref[ex:4.3]{Example~\ref{ex:4.3}}),
namely a pro-$p$ group acting on a pro-$p$ tree does not have
to be isomorphic to the fundamental pro-$p$ group of
a profinite connected graph of finite $p$-groups 
(coming from the stabilizers).
Moreover, the fundamental pro-$p$ group of a profinite graph
of pro-$p$ groups does not have to split over some edge stabilizer
as an amalgamated free pro-$p$ product or
as a pro-$p$ $\operatorname{HNN}$-extension 
(the reason is that by deleting an edge
of a profinite graph one may destroy its compactness).
These two facts are usually the major obstacles for proving
subgroup theorems of free constructions in the category of
pro-$p$ groups.

We show that the two Bass-Serre theory principal results
mentioned above hold for infinite finitely generated
pro-$p$ groups acting \emph{virtually freely} on pro-$p$ trees, 
\textit{i.e.} such that the restriction of the action 
to some open subgroup is free.
Such a group is then virtually free pro-$p$. 

\begin{bigthm}\label{t:treeacting_intro}
Let $G$ be an infinite finitely generated pro-$p$ group acting
virtually freely on a pro-$p$ tree $T$.
Then:

\begin{enumerate}
  \item[\textup{(a)}] $G$ splits over some edge stabilizer either as
  an amalgamated free pro-$p$ product or as
  a pro-$p$ $\operatorname{HNN}$-extension;

 \item[\textup{(b)}] $G$ is isomorphic to 
the fundamental pro-$p$ group 
of a finite connected graph of finite $p$-groups
whose edge and vertex groups are isomorphic to
the stabilizers of some edges and vertices of $T$.
\end{enumerate}
\end{bigthm}

This theorem is a pro-$p$ analogue of the description of
finitely generated virtually free discrete groups proved by
Karrass, Pietrovski and Solitar \cite[Thm.~1]{KPS:73}.
In the characterization of discrete virtually free groups
Stallings' theory of ends played a crucial role.
In fact the proof of the theorem of Karrass, Pietrovski and Solitar uses
a celebrated theorem of Stallings \cite[4.1, p.~127]{Stallings:70},
according to which
every finitely generated group
with more than one end 
splits over a finite group either 
as an amalgamated free product or 
as an $\operatorname{HNN}$-extension.
Note that a theory of ends has not been developed in the pro-$p$ situation,
although it has been initiated by Korenev \cite{Korenev:04}.

We prove
\hyperref[t:treeacting_intro]{Theorem~\ref{t:treeacting_intro}}
using purely combinatorial pro-$p$ group methods.
We should also say that in contrast to the classical theorem from
Bass-Serre theory the finite graph 
in item (b) is not $G\backslash T$.
Our finite graph 
is constructed in a special way by first modifying $T$ 
without loosing the essential information of the action
(\textit{cf.} \hyperref[l:tree_tech]{Lemma~\ref{l:tree_tech}}).

%\newpage
As a corollary of 
\hyperref[t:treeacting_intro]{Theorem~\ref{t:treeacting_intro}}
we deduce the following subgroup theorem.

\begin{bigthm}\label{t-subgrouptheorem_intro}
Let $G$ be the
fundamental pro-$p$ group of a 
finite connected graph of finite $p$-groups.
If $H$ is any finitely generated pro-$p$ subgroup of $G$,
then $H$ is the fundamental pro-$p$ group of
a finite connected graph of finite $p$-groups which are
intersections of $H$ with some conjugates of 
vertex or edge groups of $G$.
\end{bigthm}

Moreover, as an application of
\hyperref[t:treeacting_intro]{Theorem~\ref{t:treeacting_intro}}
we obtain the following result.

\begin{bigthm} \label{t:freeorabelian_intro}
Let $G=A\amalg_{C} B$ be a free pro-$p$ product of
$A$ and $B$ with procyclic amalgamating subgroup $C$.
Suppose that
the centralizer in $G$ of 
$C$ is a free abelian pro-$p$ group and contains $C$ as a direct factor.
If each $2$-generated pro-$p$ subgroup of $A$
and each $2$-generated pro-$p$ subgroup of $B$
is either a free pro-$p$ group or a free abelian pro-$p$ group
then so is each $2$-generated pro-$p$ subgroup of $G$.
\end{bigthm}

%\newpage
This is a pro-$p$ version of a fundamental classical result of
G. Baumslag \cite[Thm.~2]{Baumslag:62}.
Note that our theorem also generalizes the pro-$p$ version of
a result of B. Baumslag
\cite[p.~601]{BBaumslag:68}
for free products with cyclic amalgamations 
whose amalgamating subgroups are malnormal in both factors, 
and also
a recent result 
of Kochloukova and Zalesskii
\cite[Thm.~7.3]{KZ:11}.
A simple example
not covered by previous results in the literature
is illustrated in
\hyperref[ex:demushkin]{Example~\ref{ex:demushkin}}.

To prove 
\hyperref[t:freeorabelian_intro]{Theorem~\ref{t:freeorabelian_intro}}
we consider the standard pro-$p$ tree $T$
on which $G$ acts naturally; 
then, for any $2$-generated pro-$p$ subgroup $L$ of $G$, 
we decompose the pair $(L, T)$ as an
inverse limit of pairs $(L_n,T_n)$ satisfying the hypothesis of
\hyperref[t:treeacting_intro]{Theorem~\ref{t:treeacting_intro}}.

%%%%%%%%%%%%%%%%%%%%%%%%%%%%%%
\medskip \bigskip

\noindent \textbf{Notation.} \label{notation}
Throughout this paper, $p$ is a fixed but arbitrary prime number.
The additive group of the ring of $p$-adic integers is $\mathbb{Z}_p$;
the natural numbers, $\mathbb{N}$.
The cardinal number of a finite set $X$ is denoted by $|X|$.

For any elements $x$ and $y$ in a pro-$p$ group $G$
we shall write $y^x \!:= x^{-1} y x$.
Unless otherwise noted,
all groups are pro-$p$, subgroups are closed, and
maps are continuous. 
For a subset $A$ of $G$ we denote by $\langle A\rangle$ 
the subgroup of $G$ (topologically) generated by $A$
and by $A^G$ the normal closure of $A$ in $G$, \textit{i.e.},
the smallest closed normal subgroup of $G$ containing $A$.
By $d(G)$ we denote the smallest cardinality of a
generating subset of $G$.
The Frattini subgroup of $G$ will be denoted by $\Phi(G)$.
By $\operatorname{tor}(G)$ we mean the set
of all torsion elements of $G$.
If a pro-$p$ group $H$ is isomorphic to $G$, 
then we write $H\cong G$.

For a pro-$p$ group $G$ acting continuously
on a space $\Omega$ we denote
the set of all points of $\Omega$ fixed under $G$ by $\Omega^G$,
and for each $x$ in $\Omega$ the {point stabilizer} by $G_x$.
\emph{We define
$\widetilde{G}:=\langle{ G_x\mid x\in \Omega}\rangle$}.
The orbit set is designated
by $G\backslash\Omega$,
since group actions are assumed to be left actions.

The empty set $\emptyset$
is a profinite graph which is not connected,
in particular, it is not a pro-$p$ tree.
%A profinite graph other than a single vertex is called
%a \emph{non-trivial} profinite graph.
A profinite graph is {\em non-trivial} provided it contains more than one vertex. %wh
%The inverse limit $\varprojlim_i X_i$ of compact spaces 
%is {\em strict} provided all bonding maps $X_i\to X_j$ are onto.%wh

The rest of our notation is very standard and basically follows
the works of Ribes and Zalesskii
\cite{RZ:00a}
and
\cite{RZ:10}.

%%%%%%%%%%%%%%%%%%%%%%%%%%%%%%
%\newpage
\phantomsection
\section{Preliminary Results}
\label{s:prelim}
%%%%%%%%%%%%%%%%%%%%%%%%%%%%%%

In the current section we collect properties of
amalgamated free pro-$p$ products,
pro-$p$ $\operatorname{HNN}$-extensions and
pro-$p$ groups acting on pro-$p$ trees to be used in the paper.
Further information on this subject can be found in
\cite{RZ:00a} and \cite{RZ:10}.

First, an amalgamated free pro-$p$ product
$G=A\amalg_CB$ is \emph{non-fictitious}
if $C$ is a proper subgroup of both $A$ and $B$.
Unless differently stated \emph{we shall consider exclusively
non-fictitious amalgamated free pro-$p$ products}
and we shall make use of the facts 
established by Ribes
\cite{Ribes:71} that a free pro-$p$ product
with either procyclic or finite amalgamating subgroup
is always \emph{proper}, \textit{i.e.},
the factors $A$ and $B$ embed in $G$ via the natural maps.
Second,
a pro-$p$ $\operatorname{HNN}$-extension 
$G=\operatorname{HNN}(H,A,B,f,t)$
is \emph{proper} if the natural map from $H$ to $G$ is injective.

\begin{thm} %2.1
\label{t:afpproperties}
Let $G=G_1\amalg_H G_2$ be a proper amalgamated free pro-$p$
product of pro-$p$ groups.
\begin{itemize}
\item[\textup{(a)}] \textup{(\protect{\cite[Thm.~4.2(b)]{RZ:00a}})}
Let $K$ be a finite subgroup of $G$.
Then $K\subseteq G_i^{g}$ for some $g\in G$ and for some
$i=1$ or $2$. %\in\{1,2\}

\item[\textup{(b)}] \textup{(\protect{\cite[Thm.~4.3(b)]{RZ:00a}})}
Let $g\in G$. Then

$G_i\cap G_j^{g}\subseteq H^{b}$

for some $b\in G_i$, whenever $1\le i\neq j\le 2$ or $g\not\in G_i$.
\end{itemize}
\end{thm}

\begin{thm} %2.2
\label{t:hnnproperties}
Let $G=\operatorname{HNN}(H,A,t)$ be a proper
pro-$p$ $\operatorname{HNN}$-extension.
\begin{itemize}
\item[\textup{(a)}]
\textup{(\protect{\cite[Thm.~4.2(c)]{RZ:00a}})}
Let $K$ be a finite subgroup of $G$.
Then $K\subseteq H^{g}$ for some $g\in G$.

\item[\textup{(b)}]
\textup{(\textit{cf.} \protect{\cite[Thm.~4.3(c)]{RZ:00a}})}
Let  $g\in G $. Then
%\[
$H\cap H^{g}\subseteq A^{b}$
%\]
for some $b\in H\cup tH$, whenever $g\not\in H$.
\end{itemize}
\end{thm}

Next, we recollect 
fundamental results
from the theory of pro-$p$ groups acting on pro-$p$ trees
in the succeeding theorem.
We note that the two previous results are simple consequences
of it.
Recall first that
for a pro-$p$ group $G$ acting on a pro-$p$ tree $T$,
the closed subgroup generated by all vertex stabilizers
is denoted by $\widetilde{G}$; besides, 
the (unique) smallest pro-$p$ subtree of $T$ containing
two vertices $v$ and $w$ of $T$ is denoted by $[v,w]$
and called the geodesic connecting $v$ to $w$ in $T$
(\textit{cf.} \cite[p.~83]{RZ:00a}).

\begin{thm} %2.3
\label{t:trees1}
Let $G$ be a pro-$p$ group acting on a pro-$p$ tree $T$.
\begin{itemize}
 \item[\textup{(a)}] \textup{(\protect{\cite[Prop.~3.5]{RZ:00a}})}
$\widetilde{G}\backslash T$ is a pro-$p$ tree.

 \item[\textup{(b)}] \textup{(\protect{\cite[Cor.~3.6]{RZ:00a}})}
$G/\widetilde{G}$ is a free pro-$p$ group.

 \item[\textup{(c)}] \textup{(\protect{\cite[Cor.~3.8]{RZ:00a}})}
If $v$ and $w$ are two different vertices of $T$,
then $E([v,w])\neq\emptyset$ and
$(G_v\cap G_w)\subseteq G_e$ for every $e\in E([v,w])$.

 \item[\textup{(d)}] \textup{(\protect{\cite[Thm.~3.9]{RZ:00a}})}
If $G$ is finite, then $G=G_v$ for some $v\in V(T)$.
 \end{itemize}
\end{thm}

%\newpage
Now we quote three results to be referred to in
\hyperref[s:virt_freely]{Section~\ref{s:virt_freely}}.
%only.

%\newpage
\begin{prop}[\protect{\textnormal{\textit{cf.}}
\cite[Thm.~5.6]{Melnikov:90}, \cite[Thm.~3.6]{Zalesskii:96}}] %2.4
\label{p:freeedgeaction}
Let $G$ be a pro-$p$ group acting on a pro-$p$ tree $T$
with trivial edge stabilizers.
If there exists a continuous section
$\sigma\colon G\backslash V(T)\longrightarrow V(T)$,
then $G$ is isomorphic to the free pro-$p$ product
\[
\left(\coprod_{\dot w\in G\backslash V(T)}G_{\sigma(\dot w)}\right)
\amalg \left(G/\langle{G_v\mid v\in V(T)}\rangle\right) \, .
\]
\end{prop}

%\newpage
\begin{prop}[\protect{\cite[Thm.~1.1]{Scheiderer:99}}] %2.5
\label{p:scheiderer}
Let $G$ be a finitely generated pro-$p$ group which
contains an open free pro-$p$ subgroup of index $p$.
Then $G$ is isomorphic to a free pro-$p$ product
\[
F_0\amalg(C_1\times F_1)\amalg\cdots\amalg(C_m\times F_m)
\]
where $m\geq 0$, 
the $F_i$ are free pro-$p$ groups of finite rank
and
the $C_i$ are cyclic groups of order $p$.
\end{prop}
%cited in 

\begin{cor}[\protect{\cite[Cor.~1.3(a)]{Scheiderer:99}}] %2.6
\label{c:finiteconjugacy}
Every pro-$p$ group which contains an open
free pro-$p$ subgroup of finite rank
has, up to conjugation, only a finite number
of finite subgroups.
\end{cor}
%cited in 

The definition of the fundamental pro-$p$ group
of a connected profinite graph of pro-$p$ groups
is quite involved
(see \cite[1.7 and 2.1]{ZM:90}).
However, the fundamental pro-$p$ group
$\Pi_1(\mathcal{G}, \Gamma)$
of a finite connected graph of finitely generated pro-$p$ groups
$(\mathcal{G}, \Gamma)$
can be defined as the pro-$p$ completion
of the abstract (usual) fundamental group
$\Pi_1^{abs}(\mathcal{G},\Gamma)$,
by using the fact that
every subgroup of finite index
in a finitely generated pro-$p$ group is open.
We shall need only this case throughout the paper.
Thus $\Pi_1(\mathcal{G}, \Gamma)$
has the following pro-$p$ presentation.
\begin{center}
\begin{tabular}{l l l}
Generators:
&generators of $\mathcal{G}(v)$,
&$v\in V(\Gamma)$\\
&$t_e$,
&$e\in E(\Gamma)$\\
Relations:
&relations of $\mathcal{G}(v)$,
&$v\in V(\Gamma)$\\
&$\partial_{0,e}(g)=t_e\partial_{1,e}(g)t_e^{-1}$,
 \ for all $g\in \mathcal{G}(e)$,
&$e\in E(\Gamma)$\\
&$t_e=1$,
&$e\in E(T)$
\end{tabular}
\end{center}
where,
the generators and relations of each
vertex group $\mathcal{G}(v)$ are taken from any chosen
pro-$p$ group presentation of $\mathcal{G}(v)$,
the letters $t_e$ are different from
the generators of all vertex groups,
the maps
$\partial_{0,e}\colon\mathcal{G}(e)\to \mathcal{G}(d_0(e))$ and
$\partial_{1,e}\colon\mathcal{G}(e)\to \mathcal{G}(d_1(e))$
are the given monomorphisms
from each edge group to its initial and terminal vertex group,
and
$T$ is any chosen maximal subtree of $\Gamma$. 

%%%%%%%%%%%%%%%%%%%%%%%%%%%%%%
%\newpage
\phantomsection
\section{Groups acting virtually freely on trees}
\label{s:virt_freely}
%%%%%%%%%%%%%%%%%%%%%%%%%%%%%%

The goal of the present section is to establish Theorems 
\hyperref[t:treeacting_intro]{\ref{t:treeacting_intro}}
and
\hyperref[t-subgrouptheorem_intro]{\ref{t-subgrouptheorem_intro}}.

\begin{lemma} %3.1
\label{l:actionfree}
%\small
Let $G$ be a finitely generated pro-$p$ group
acting on a connected profinite graph $\Gamma$.
Suppose there exists a connected profinite subgraph
$\Delta$ of $\Gamma$ such that $G\Delta=\Gamma$.
Then there exists a generating set $S$ of $G$
such that $|S|=d(G)$ and
$\Delta\cap s\Delta\neq\emptyset$ for each $s$ in $S$.
\end{lemma}
%cited in 3.2

\begin{proof}
It is enough to prove the lemma under the
additional assumption that $G$ is non-trivial elementary abelian.
Indeed, using ``bar'' to denote passing to the quotients modulo
the Frattini $\Phi(G)$,
by assumption there exists 
a subset $Z$ of $G$ such that
$\overline{Z}$ is a generating set of $\overline{G}$,
$|Z|=d(\overline{G})=d(G)$, %=|\overline{Z}|
and with the property that
$\overline{\Delta}\cap\overline{z}\overline{\Delta}\neq\emptyset$
for each $z$ in $Z$.
Then there exist $f_z$ in $\Phi(G)$ such that
$\{f_z z\mid z\in Z\}$ is a desired set
$S$ of generators of $G$.
Of course, if $G$ is trivial then we take $S=\emptyset$
and the lemma is vacuously true.

So henceforth $G$ is a non-trivial finite elementary abelian group.
We proceed by induction on $d(G)$.
Suppose first $d(G)=1$.
Since the union $\bigcup_{g\in G}g\Delta$ cannot be disjoint
(because $\Gamma$ and each $g\Delta$ is connected),
there exist two distinct elements $g_1$ and $g_2$ in $G$
such that $g_1\Delta\cap g_2\Delta\neq\emptyset$.
Taking $S$ as $\{g_1^{-1}g_2\}$,
the conclusion of the lemma holds.
Suppose now $d(G)\ge2$.
Pick any element $t$ in $G-\Phi(G)$, and %\smallsetminus
let ``bar'' denote passing to the quotient modulo $\langle t\rangle$.
By induction, there is a subset $Y$ of $G$ such that
$\overline{Y}$ is a generating set of $\overline{G}$,
$|Y|=d(\overline{G})=d(G)-1$,
and for each $y$ in $Y$ we have
$\Delta\cap t_y y\Delta\neq\emptyset$
for some element $t_y$ in $\langle t\rangle$.
Set $W=\{t_y y\mid y\in Y\}$.
Since $\langle W\rangle\Delta$ is connected
and $\langle t\rangle\langle W\rangle\Delta=\Gamma$,
by the basis of induction
there exists $t'$ in $\langle t\rangle$
such that
$\langle W\rangle\Delta\cap t'\langle W\rangle\Delta\neq\emptyset$.
This means that there exist $h_1$ and $h_2$ in $\langle W\rangle$
such that $h_1\Delta \cap t'h_2\Delta\neq \emptyset$,
so $S$ equals $W\cup\{h_1^{-1}t'h_2\}$ satisfies the assertion
of the lemma.
\end{proof}

%\newpage
Recall our ``tilde'' notation for
the closed subgroup generated by all stabilizers.

\begin{lemma} %3.2
\label{l:action}
Let $G$ be a finitely generated pro-$p$ group
acting on a pro-$p$ tree $T$.
Suppose there exists a pro-$p$ subtree $D$ of $T$
such that $GD=T$.
Then there exists a subset $X$ of $G$ which is a free generating set
of a retract of $G/\widetilde{G}$ in $G$ such that
$|X|=d(G/\widetilde{G})$ and
$D\cap xD\neq\emptyset$ for each $x$ in $X$.
\end{lemma}
%cited in 3.3

\begin{proof}
Let ``bar'' denote passing to the quotient modulo $\widetilde G$.
By
\hyperref[t:trees1]{Theorem~\ref{t:trees1}(a)}
the quotient graph  $\overline{T}$ is a pro-$p$ tree.
Applying
\hyperref[l:actionfree]{Lemma~\ref{l:actionfree}}
to $\overline{G}$ acting on $\overline{T}$
yields a subset $S$ of $G$ such that
$\overline{S}$ is a generating set of $\overline{G}$,
$|S|=d(\overline{G})$, and
for each $s$ in $S$ we have $D\cap k_s sD\neq\emptyset$
for certain $k_s$ in $\widetilde G$.
Set $X=\{k_s s\mid s\in S\}$. % and $H=\langle X\rangle$.
Finally observe that by
\hyperref[t:trees1]{Theorem~\ref{t:trees1}(b)},
$\overline{G}$ is a free pro-$p$ group, so
$X$ freely generates a retract.
\end{proof}

%\newpage
In a connected non-trivial profinite graph,
every vertex is the initial or terminal vertex of some edge,
provided the set of edges is compact
(\textit{cf.} \cite[Lemma 2.14]{ZM:89}).
The following analogue result concerns stabilizers.
For the purposes of this paper,
it suffices us to endow
the set of all closed subgroups of a profinite group $G$
with a topology such that:
$G$ acts continuously on it by conjugation;
and, the set of all closed subgroups contained
in a given closed subgroup of $G$ is closed in it.
If $G=\varprojlim G_i$ with finite discrete groups $G_i$,
then it is natural and sufficient
to consider the topology of the projective limit
of the finite discrete sets of all closed subgroups of $G_i$.

\begin{lemma} %NEW3.3
\label{l:edges-vertices}
Let $G$ be a profinite group
acting on a connected non-trivial profinite graph.
Suppose that the set of all edge stabilizers is compact
in the space of all closed subgroups of $G$.
Then
every vertex stabilizer contains an edge stabilizer.
\end{lemma}
%cited in 3.4

\begin{proof}
Let $G$ act on a connected non-trivial profinite graph $\Gamma$.
Let $\mathcal K$ denote the set of all edge stabilizers.
Write $\Gamma=\varprojlim_{i\in I}\Gamma_i$,
where $I$ is a right directed ordered set,
$\Gamma_i$ are finite (connected) non-trivial quotient graphs of $\Gamma$,
$G$ acts on each $\Gamma_i$,
and each projection $\Gamma\to\Gamma_i$ is a
$G$-morphism of profinite graphs
(\textit{cf.} \cite[Lemma~5.6.4(a)]{RZ:10}).
Let $v\in V(\Gamma)$.
For each $i$ in $I$,
let $v_i$ be the projection of $v$ in $\Gamma_i$ and
consider the set $\Upsilon_i$
of all stabilizers of edges of $\Gamma$ contained in $G_{v_i}$.
Evidently,
$\{\Upsilon_i\}_{i\in I}$ is a 
family of non-empty closed subsets
in $\mathcal K$ 
with the finite intersection property.
Indeed: first, since $|\Gamma_i|>1$,
there exists an edge having initial or terminal vertex $v_i$
and with preimage in $E(\Gamma)$,
so $\Upsilon_i\neq\emptyset$;
second, since $G_{v_i}$ is closed in $G$,
the set of all closed subgroups of $G$ contained in $G_{v_i}$
is closed in the set of all closed subgroups of $G$,
hence $\Upsilon_i$ is closed in $\mathcal K$;
third, since $\Upsilon_j\supseteq \Upsilon_i$ whenever $j\le i$,
for any finite subset $J$ of $I$ there exists $i_J$ in $I$
such that
$\bigcap_{j\in J} \Upsilon_j\supseteq\Upsilon_{i_J}\neq\emptyset$.
So, by the compactness of $\mathcal K$
we have that $\bigcap_{i\in I} \Upsilon_i\neq\emptyset$,
and since $\bigcap_{i\in I} G_{v_i}= G_v$,
it follows that
there exists an edge stabilizer contained in $G_v$.
\end{proof}

%\newpage
Let $G$ be a pro-$p$ group.
Recall that $G$ acts
\emph{faithfully} on a pro-$p$ tree $T$ if the kernel of the action is trivial;
and $G$ acts \emph{irreducibly} on $T$
if $T$ contains no proper $G$-invariant pro-$p$ subtree.
Also,
we remind the reader
that $G$ acts \emph{virtually freely}
on a space $\Omega$ if some open subgroup of $G$ acts freely on $\Omega$
by restriction. 

\begin{lemma} %3.4
\label{l:tree_tech}
%\small
Let $G$ be a non-trivial finitely generated pro-$p$ group acting
faithfully, irreducibly and virtually freely on a pro-$p$ tree $T$.
Then there exist
a quotient pro-$p$ tree $D$ on which $G$ acts, 
an edge $e$ of $E(D)$,
a finite subset $V$ of $V(D^{G_e})$, and
a finite subset $X$ of $G$ such that:
\begin{itemize}
\item[\textup{(a)}]
$G$ acts faithfully and irreducibly on $D$;
\item[\textup{(b)}]
$G_e$ equals the stabilizer of some edge of $T$,
the stabilizer of every edge of $D$ is conjugate to $G_e$,
and $GD^{G_e}=D$;
\item[\textup{(c)}]
$V$ has at most one element of each $G$-orbit in $V(D)$;
\item[\textup{(d)}]
$X$ freely generates a free pro-$p$ subgroup of $G$
such that 
$G=\langle G_v , X\mid v\in V\rangle$ and
$\langle X\rangle \cap \langle G_v\mid v\in V\rangle^{G}=\{1\}$;
\item[\textup{(e)}]
for each $x$ in $X$, there exists $v$ in $V$ such that
$
xG_ex^{-1}\subseteq G_v \, ;
$
\end{itemize}
\end{lemma}
%cited in 3.6

\begin{proof}
First we construct the pro-$p$ tree $D$.
Since $T$ is non-trivial,
%because G is non-trivial and the action is faith.
we take an edge $a$ of $T$ with stabilizer $G_a$ of minimal order.
Let $\Sigma_a$ denote the set of all finite subgroups $L$ %(non-trivial)
of $G$ that are not conjugate to any subgroup of $G_a$.
Since $G$ is finitely generated,
\hyperref[c:finiteconjugacy]{Corollary~\ref{c:finiteconjugacy}}
says that there exist up to conjugation
only finitely many finite subgroups in $G$;
in particular, there is a finite subset $\Xi$ of
$\Sigma_a$ such that $\Sigma_a=\{L^g\mid L\in\Xi, \ g\in G\}$.
Therefore,
the union $T_{\Sigma_a}$ of all
fixed points $T^L$
for $L\in\Sigma_a$ can be represented in the form
$T_{\Sigma_a}=\bigcup_{L\in\Xi}GT^L$ and is hence a
$G$-invariant profinite subgraph of $T$.
\emph{Let $D$ be the profinite graph obtained by collapsing
the distinct connected components of $T_{\Sigma_a}$ in $T$
to distinct points};
if $T_{\Sigma_a}$ is empty then $D$ is nothing but $T$.
By the pro-$p$ version of
\cite[Proposition, p.~486]{Zalesskii:89}, %1.6.9
$D$ is a $p$-simply connected profinite graph; 
hence $D$ is a pro-$p$ tree %1.8.2
on which $G$ acts naturally.

Henceforth,
let us use ``bar'' to denote the collapsing procedure.
Note first that
the stabilizers of vertices in $D$ may well be infinite
if $T_{\Sigma_a}$ is non-empty.
Moreover,
if $m\in T-T_{\Sigma_a}$ then
$G_{\overline m}$ equals $G_m$ and
is contained, up to conjugation, in $G_a$
(because $Gm\cap T_{\Sigma_a}=\emptyset$
and $G_m\not\in\Sigma_a$).

Now, the properties in item (a) come from the original action.
Indeed, 
suppose that $T_{\Sigma_a}\neq\emptyset$ and
that the action of $G$ on $D$ were not irreducible.
Since $D$ is obtained by collapsing pro-$p$ subtrees,
the preimage of a proper $G$-invariant pro-$p$ subtree of $D$
would be a proper $G$-invariant pro-$p$ subtree of $T$;
a contradiction.
So, $G$ acts irreducibly on $D$.
Besides, 
suppose that $g$ in $G$ acts trivially upon all of $D$.
Then, in particular, 
$g\in G_{\overline{a}}=G_a$.
Hence
$G_a$ contains the kernel of the action of $G$ upon
$D$ which, by
\cite[Thm.~3.12]{RZ:00a},
must act trivially on $T$.
Therefore $G$ also acts faithfully on $D$.

\emph{Let $e=\overline{a}$}.
All assertions of item (b) 
follow from the second paragraph,
with the last one making use of the previous lemma.
In fact,
clearly $G_e=G_a$. 
%For proving (b) first put \emph{Let $e=\overline{a}$}.
%Then clearly $G_e=G_a$. 
Let $b\in E(T)-T_{\Sigma_a}$.
Since
$G_{\overline{b}}$ is contained, up to conjugation, in $G_a$,
from the minimality of $G_a$ we have that
$G_{\overline{b}}$ is conjugate to $G_e$.
Lastly, since $G$ acts continuously by conjugation on the
space of all its closed subgroups,
the conjugacy class of $G_e$ is compact, and
\hyperref[l:edges-vertices]{Lemma~\ref{l:edges-vertices}}
gives that
the stabilizer of {any} vertex in $D$ contains, 
up to conjugation, $G_e$;
that is, $GV(D^{G_e})=V(D)$.
Therefore we conclude that $GD^{G_e}=D$.

We come to prove (c), (d) and (e).
From
\hyperref[l:action]{Lemma~\ref{l:action}},
consider a finite \emph{subset $X$} of $G$,
which is a free generating set
of a retract of $G/\langle{G_v\mid v\in V(D)}\rangle$ in $G$ such that
$D^{G_e}\cap xD^{G_e}\neq\emptyset$ for each $x$ in $X$.
Clearly,
$G=\langle{G_v, X\mid v\in V(D)}\rangle$,
$\langle X\rangle \cap \langle{G_v\mid v\in V(D)}\rangle^G=\{1\}$,
and $v\in D^{G_e}\cap xD^{G_e}$
can be written as $G_e\cup xG_e x^{-1}\subseteq G_v$.
Now we note that,
by choosing a complete set of representatives of elements
in a generating set of a pro-$p$ group
of each conjugacy class
we still obtain a generating set
(because the pro-$p$ Frattini quotient is abelian).
%%and \cite[Cor. 2.8.5]{RZ:10}
%Also, we replace each $x$ in $X$ by its inverse.
%%%%%to obtain G_e^x\subseteq \cup G_v

Next, from $GV(D^{G_e})=V(D)$, it is clear that
a vertex which is the image of a vertex in $T-T_{\Sigma_a}$ 
has stabilizer conjugate to $G_e$.
On the other hand, we claim that
there is only a finite subset of vertices in $D$ up to translation
with stabilizers that are not conjugate to $G_e$.
%So ``(c) holds''.
In fact,
the number of connected components of $T_{\Sigma_a}$
up to translation is at most $|\Xi|$;
for, otherwise,
we could find two vertices $v$ and $w$ in 
distinct connected components %of $T_{\Sigma_a}$
up to translation with conjugate stabilizers,
hence $G_v$ would stabilize
the geodesic from $v$ to a translation of $w$
(see 
\hyperref[t:trees1]{Theorem~\ref{t:trees1}(c)}),
and these two vertices would be in the same
connected component, a contradiction.
Therefore, there exists a finite subset of vertices $W$,
which we may assume to be contained in $D^{G_e}$, such that
$G=\langle{G_v, X\mid v\in W}\rangle$ 
and
$\langle X\rangle \cap \langle{G_v\mid v\in W}\rangle^G=\{1\}$.
%So ``(d) holds''.
Finally,
to define a desired \emph{set $V$},
for each $x$ in $X$
we modify $W$ by adding, if necessary,
or replacing 
%(or, if necessary, by adding) %so $|W|\le |\Xi|+|X|$
a vertex $v$ in $W$ by a vertex $v_x$
in $D^{G_e}\cap xD^{G_e}$, 
whenever $Gv=Gv_x$.
Thus we see that (c), (d), and (e) all hold.
\end{proof}

The last ingredient to prove
\hyperref[t:treeacting_intro]{Theorem~\ref{t:treeacting_intro}}
is the following version of the Schreier index formula.

\begin{lemma} %3.5
\label{l:rankformula}
%\small
Let $G$ be a finitely generated pro-$p$ group
acting on a pro-$p$ tree $T$.
Suppose that all vertex stabilizers are finite and
all edge stabilizers are pairwise conjugate.
Assume further that there exist an edge $e$ of $T$,
a finite subset $U$ of $V(T^{G_e})$,
and a finite subset $X$ of $G$ such that:
\begin{itemize}
\item[\textup{(i)}] 
$U$ has exactly one element of each $G$-orbit in $V(T)$;
\item[\textup{(ii)}]  
$X$ freely generates a free pro-$p$ subgroup of $G$
such that
$G=\langle G_u , X\mid u\in U\rangle$ and
$\langle X\rangle \cap \langle G_u\mid u\in U\rangle^{G}=\{1\}$;
\item[\textup{(iii)}]
for each $x$ in $X$, there exists $u$ in $U$ such that
$
xG_ex^{-1}\subseteq G_u \, .
$
\end{itemize}
If $F$ is a free pro-$p$ open normal subgroup of $G$, then 
\[
\operatorname{rank}(F)-1=
[G:F]\left(\frac{|U|+|X|-1}{|G_e|}-\sum_{u\in U} \frac{1}{|G_u|}\right) \, .
\]
\end{lemma}
%cited in 3.6

\begin{proof}
We proceed by induction on the index $[G:F]$.
If $[G:F]=1$, then $G=\langle X\rangle$ from hypothesis~(ii),
%and so $\operatorname{rank}(F)=|X|$; 
and there is nothing to prove.
%Obviously $F\neq G$, from hypothesis~(ii); so,
Otherwise,
let us consider the preimage $N$ in $G$ of a central subgroup
of order $p$ of $G/F$.

\medskip
\noindent Case 1.
\textit{$N\cap G_e=\{1\}$.}
\medskip

According to
\hyperref[p:scheiderer]{Proposition~\ref{p:scheiderer}}
we have
$
N=F_0 \amalg (C_1\times F_1 )\amalg\cdots\amalg (C_m\times F_m),
$
where $m\geq 0$, 
the $F_i$ are free pro-$p$ groups of finite rank
and the $C_i$ are cyclic groups of order $p$.
If $m=0$, then
the Schreier index formula for $F$ in $N$ together with
the induction hypothesis for the rank of $N$ yields
the desired formula for the rank of $F$.

Suppose then $m\ge 1$.
We claim that
\begin{equation}
\label{e:scheiderer}
N=F_0\amalg C_1 \amalg \ldots \amalg C_m \ ,
\end{equation}
with $F_0$ a free pro-$p$ subgroup of $F$.
Indeed, let us first prove that
each non-trivial torsion element
$s$ in $N$ generates a self-centralized subgroup of $N$. %wh
Take any element $g$ in $N$ centralizing $s$.
Since
$\langle s\rangle$ is the stabilizer of some vertex $w$ of $T$
(by \hyperref[t:trees1]{Theorem~\ref{t:trees1}(d)}),
the element $s$ also stabilizes $gw$.
If $gw\ne w$, then from
\hyperref[t:trees1]{Theorem~\ref{t:trees1}(c)}
$s$ stabilizes
the geodesic $[gw,w]$;
however, since $N\cap G_e=\{1\}$,
the element $s$ cannot stabilize any edge.
Hence $gw=w$, and therefore $g$ is a power of $s$.
Now, furthermore, the last assertion of our claim
follows from 
\hyperref[t:afpproperties]{Theorem~\ref{t:afpproperties}(c)}.

So, by the Kurosh subgroup theorem
(\textit{cf.} \protect{\cite[Thm.~9.1.9]{RZ:10}})
we have
\begin{equation}
\label{e:rankF}
\operatorname{rank}(F)=
[N:F]\operatorname{rank}(F_0)+(1+m[N:F]-[N:F]-m) \, .
\end{equation}
Now, taking into account
\hyperref[t:trees1]{Theorem~\ref{t:trees1}(d)}
and that $G$ also acts upon
the conjugacy classes of subgroups of order $p$, 
%because of hypothesis~(i)
we rearrange the free pro-$p$ factors in
equation~(\ref{e:scheiderer}) as
\[
N=F_0\amalg\coprod_{u\in U'}\left(\coprod_{r_u\in G/NG_u}
(N\cap G_{u})^{r_u}\right) \, ,
\]
where $U':=\{u\in U\mid N\cap G_{u}\neq \{1\} \}$.
Since $FG_u=NG_u$ for every $u$ in $U'$,
using this rearranged decomposition and comparing it with
equation~(\ref{e:scheiderer})
we find
\begin{equation}
\label{e:AinLemma}
m=\sum_{u\in U'}|G/FG_u|=[G:F]\sum_{u\in U'}\frac{1}{|G_u|}
\, .
\end{equation}

If $N=G$, then $G_e=\{1\}$, 
$\operatorname{rank}(F_0)=|X|$,
and $|U|=m$.
So, equation~(\ref{e:rankF}) 
becomes exactly the needed one.

Suppose now that $N\not= G$.
Then the product $p\operatorname{rank}(F_0)$ can be computed
by observing that passing to the quotient modulo
$\langle\operatorname{tor}(N)\rangle$
and indicating it by ``bar'' we have
$\operatorname{rank}(\overline{F})=
\operatorname{rank}(F_0)$,
so that using $[G:F]=p[\overline{G}:\overline{F}]$
the induction hypothesis yields
\begin{eqnarray*}
p\operatorname{rank}(F_0)&=&p\operatorname{rank}(\overline{ F})\\
&=&p[\overline{G}:\overline{F}]
\left(\frac{|U|-1}{|\overline{ G_e}|}
-\sum_{u\in U} \frac{1}{|\overline{G_u}|}\right)
+p\\
&=&[G:F] 
\left(\frac{|U|-1}{|G_e|}
-\sum_{u\in U'} \frac{1}{|\overline{G_u}|}
-\sum_{u\in U-U'} \frac{1}{|G_u|}\right)
+p\\
&=&[G:F]
\left(\frac{|U|-1}{|G_e|}
-\sum_{u\in U'} \frac{p}{|G_u|}
-\sum_{u\in U-U'} \frac{1}{|G_u|}\right)
+p
\end{eqnarray*}
(we used $N\cap G_e =\{1\}=N\cap G_u$
for all $u\in U-U'$ and $|N\cap G_u|=p$ for all $u\in U'$
to obtain the last equality).
Inserting this expression and the expression for $m$
from equation~(\ref{e:AinLemma}) into equation~(\ref{e:rankF})
yields the claimed formula for $\operatorname{rank}(F)$.

\medskip
\noindent Case 2.
\textit{$N\cap G_e \not= \{1\}$.}
\medskip

We claim that $N\cap G_e$ is a normal subgroup of $G$ of order $p$.
Indeed, let $u\in U$.
Since all vertex stabilizers are finite, from
$(N\cap G_e)F/F\subseteq (N\cap G_u)F/F\subseteq N/F$
it follows that 
$N\cap G_e$ coincides with $N\cap G_u$ and has order $p$,
and from $G/F$ centralizing $N/F$ we have that 
$G_u$ centralizes $N\cap G_e$.
Furthermore,
by hypothesis~(iii), 
$X$ normalizes $N\cap G_e$
(since closed subsemigroups in compact groups are subgroups).
Thus, from hypothesis~(ii), the claim follows.

%By assumption we have $U\subseteq V(T^{G_e})$, 
%i.e., $G_e\subseteq \bigcap_{u\in U}G_u$. 
%Therefore, deducing from $G_u\cap F=1$ the fact that 
%$(N\cap G_u)F/F=(N\cap G_e)F/F$ 
%coincides with the central subgroup $N/F$ of order $p$,
%we note for all $u\in U$ that 
%$N\cap G_e$ is a central subgroup of $G_u$ of order $p$.
%Pick any $x\in X$. 
%Then $xG_ex^{-1}\subseteq G_{u_0}$ for some 
%$u_0\in U$ by assumption (iii) and 
%therefore $x(G_e\cap N)x^{-1}\subseteq G_{u_0}$ holds.
%Hence $x(G_e\cap N)x^{-1}\subseteq G_{u_0}\cap N$ 
%which equals to $G_e\cap N$. 
%It follows that  $X$ centralizes $N\cap G_e$. %
%So, by hypothesis~(ii), we have that 
%$N\cap G_e$ is a normal subgroup of $G$ of order $p$.

Let ``bar'' denote passing to the quotient modulo $N\cap G_e$.
Note that,
$\overline{T}$ is a pro-$p$ tree
(from hypothesis~(ii) and
\hyperref[t:trees1]{Theorem~\ref{t:trees1}(b)})
on which $\overline{G}$ acts in the obvious manner,
and all the assumptions of the lemma hold
modulo $N\cap G_e$;
moreover,
from hypothesis~(i)
it follows that
$\overline{U}$ and $U$ are in bijection.
So, the induction hypothesis yields
\begin{eqnarray*}
\operatorname{rank}(F)-1&=& \operatorname{rank}(\overline{F})-1 \\
&=&[\overline{G}:\overline{F}] 
\left(\frac{|\overline{U}|+|\overline{X}|-1}{|\overline{G}_{\overline{e}}|}
-\sum_{\overline{u}\in \overline{U}} \frac{1}{|\overline{G}_{\overline{u}}|}\right) \\
&=&\frac{[G:F]}{p}
\left(\frac{p(|U|+|X|-1)}{|G_e|}
-\sum_{u\in U} \frac{p}{|G_u|}\right) \\
&=&[G:F]
\left(\frac{|U|+|X|-1}{|G_e|}-\sum_{u\in U} \frac{1}{|G_u|}\right)
\end{eqnarray*}
as needed.
\end{proof}

\newcommand{\cG}{{\mathcal G}}
\begin{rmk}\label{euler}\rm
Let $G=\Pi_1(\cG,\Gamma,L)$ 
be the pro-$p$ fundamental group of a finite graph 
of finite $p$-groups $(\cG,\Gamma)$ and 
spanning tree $L$ of $\Gamma$. 
The result
\cite[Lemma~3.15]{ZM:89} implies for any open
free pro-$p$ subgroup $F$ of $G$ the validity of the rank formula
\[
\operatorname{rank}(F)-1=|G:F|
\left(\sum_{e\in E(\Gamma)}\frac1{|\cG(e)|}
-\sum_{v\in V(\Gamma)}\frac1{|\cG(v)|}\right) .
\]
Let us make the assumption that for all $e$ not in $L$ 
the edge groups $\cG(e)$ 
have the same cardinality and 
let $G$ act on its standard graph $T$. 
Then we can lift the spanning tree to $T$ in order to get 
a set $U$ of representatives of vertices of $V(\Gamma)$ in $T$ 
and the edges $e$ not belonging to $L$ give rise to 
a set of generators $X$ of a free pro-$p$ subgroup of $G$.
The preceding Lemma rephrases this rank formula, 
since $G$ and $T$ satisfy the premises of it and 
$\Gamma$ having $|U|$ vertices and 
$|X|$ edges not contained in a spanning tree, has number
of edges equal to $|U|+|X|-1$.
\end{rmk}

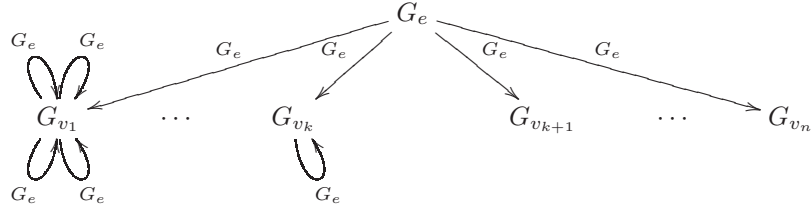
\begin{figure}
\[ 
\xymatrix{
& & &
{G_e} \ar[dlll]_{G_e} \ar[dl]_{G_e} \ar[dr]^{G_e} \ar[drrr]^{G_e}\\
{G_{v_1}} \ar@(ul,u)[]^{G_e} \ar@(u,ur)[]^{G_e}
\ar@(dl,d)[]_{G_e} \ar@(d,dr)[]_{G_e}
&\cdots &
{G_{v_ k}} \ar@(d,dr)[]_{G_e} &
&
{G_{v_ {k+1}}}
& \cdots &
{G_{v_n}}
}
\] 
\caption{A graph of groups with equal edge groups as in Remark~\ref{euler}}
\label{fig1}
\end{figure}

%\newpage
Our next two results constitute
\hyperref[t:treeacting_intro]{Theorem~\ref{t:treeacting_intro}}.
\begin{thm} %3.6
\label{t:treeacting}
An infinite finitely generated pro-$p$ group acting
virtually freely on a pro-$p$ tree 
splits over some edge stabilizer either
as an amalgamated free pro-$p$ product or 
as a pro-$p$ $\operatorname{HNN}$-extension.
\end{thm}

\begin{proof}
%\large
Let $G$ be an infinite finitely generated pro-$p$ group
acting on a pro-$p$ tree $T$
with an open normal subgroup $F$ of $G$ acting freely on $T$ 
by restriction.
The proof is by induction on $[G:F]$.
If $[G:F]=1$,
then $G$ is obviously an amalgamated free pro-$p$ product
over a trivial edge stabilizer.
Suppose then $[G:F]>1$.

By \cite[Lemma~3.11]{RZ:00a}, there exists a 
(non-empty) unique minimal $G$-invariant pro-$p$ subtree in $T$;
replacing $T$ by this subtree we may assume that
the original action of $G$ is irreducible.

\medskip
\noindent Case 1.
\textit{$G$ acts irreducibly but non-faithfully on $T$.}
\medskip

Since an open subgroup of $G$ acts freely on $T$, 
the kernel of the action must be finite. 
Hence, $G$ contains a central subgroup $C$ of order $p$,
which must act trivially on $T$ 
(\textit{cf.} \cite[Thm.~3.12]{RZ:00a}).
By induction, since $[G/C:FC/C]<[G:F]$,
the quotient group $G/C$ splits 
over the stabilizer $\overline{K}$ of an edge of $T$
either as
an amalgamated free pro-$p$ product,
$G/C=\overline{G_1}\amalg_{\overline{K}}\overline{G_2}$,
or as
a pro-$p$ $\operatorname{HNN}$-extension,
$G/C=\operatorname{HNN}(\overline{G_1},\overline{K},t)$.
But then, either $G=G_1\amalg_K G_2$ or
$G=\operatorname{HNN}(G_1,K,t)$
with $G_1$, $G_2$, and $K$ being the preimages of
$\overline{G_1}$, $\overline{G_2}$, and $\overline{K}$ in $G$,
respectively;
note that the subgroup $K$ is indeed an edge stabilizer.

\medskip
\noindent Case 2.
\textit{$G$ acts irreducibly and faithfully on $T$.}
\medskip

Readily, \textit{there exist $D$, $e$, $V$, $X$, and $G_0$ having the properties (a)--(e) of} 
\hyperref[l:tree_tech]{\textup{Lemma~\ref{l:tree_tech}}} and henceforth we consider the action of $G$ on $D$.
Let $N$ be an open normal subgroup of $G$ contained in $F$ and $\widetilde N:=\langle G_v\cap N\mid v\in V(D)\rangle$.
Then $D_N:=\widetilde N\backslash D$ is a pro-$p$ tree on which $G_N:=G/\widetilde N$ acts.
Since $G_e\cap N=\{1\}$ we find that all edge stabilizers of $D_N$ are conjugate to $G_e\widetilde N/\widetilde N\cong G_e$.
Moreover, if $N$ is ``small enough'' we can arrange that the images of the
data $V$, $e$, $X$, and $G_0$ do have the properties (a)--(e) of 
\hyperref[l:tree_tech]{\textup{Lemma~\ref{l:tree_tech}}} and $V_N:=\widetilde N \backslash V$ and
$X_N:=X\widetilde N/\widetilde N\subset G_N$ have respectively the same cardinalities as $V$ and $X$.
Let $\mathcal N$ denote the set of open normal subgroups $N$ of $G$ enjoying these
properties.
So, for each $N\in\mathcal N$,
\hyperref[l:tree_tech]{\textup{Lemma~\ref{l:tree_tech}}}
gives us the fundamental pro-$p$ group $\Pi_N$
of a finite connected graph of groups 
described as follows (\textit{cf.} \hyperref[euler]{\textup{Figure~\ref{euler}}}):
there are $|V|+1$ vertex groups;
for each $v$ in $V$ we have one vertex group $G_v\widetilde N/\widetilde N$
which is the terminal vertex group
of a single edge group $G_e\widetilde N/\widetilde N$ having initial vertex group
the distinguished vertex group $G_e\widetilde N/\widetilde N$;
for each $x$ in $X_N$ we have an edge group $G_e\widetilde N/\widetilde N$
connecting at both ends the same vertex group $G_v\widetilde N/\widetilde N$,
whenever there exists $v$ in $V$ such that
$xG_ex^{-1}\subseteq G_v$;
each morphism from an edge group to its
initial or terminal vertex group is the inclusion map.

Let $\lambda_N:\Pi_N\to G/\widetilde N$ be the canonical epimorphism
induced by universality from the map that sends identically every $G_v\widetilde N/\widetilde N$ to its isomorphic
copy in $G/\widetilde N$ and every $x\in X_N$ identically to $x\in G/\widetilde N$.
Observe that $N/\widetilde N$ is free pro-$p$ of finite rank and that the vertex stabilizers of $G_N$
acting on $D_N$ are all finite. 
We claim that the preimage $\lambda_N^{-1}(N/\widetilde N)$ is a free pro-$p$ group.
Indeed, since $\ker(\lambda_N)$ acts freely by restriction on the standard graph $S_N$ of $\Pi_N$,
$\ker\lambda_N$ is a torsion-free group. Thus $\lambda_N^{-1}(N/\widetilde N)$, as an extension of torsion-free
groups, is also torsion-free, and hence $\lambda_N^{-1}(N/\widetilde N)$ acts by restriction on $S_N$ with
trivial edge stabilizers. Taking into account that by construction $S_N/\lambda_N^{-1}(N/\widetilde N)$
is finite, 
we apply \hyperref[p:freeedgeaction]{\textup{Proposition~\ref{p:freeedgeaction}}}
to get that $\lambda_N^{-1}(N/\widetilde N)$ is a free pro-$p$ product of free pro-$p$ groups, and hence 
a free pro-$p$ group.
Now, letting $N/\widetilde N$ play the role of $F$ in the rank formula of
\hyperref[l:rankformula]{\textup{Lemma~\ref{l:rankformula}}} in comparison with the rank formula
of \hyperref[euler]{\textup{Remark~\ref{euler}}} we deduce that the restriction of
$\lambda_N$ to $\lambda_N^{-1}(N/\widetilde N)$ is an isomorphism and therefore $\lambda_N:\Pi_N\to G_N$ is
an isomorphism. 
Observing for open normal subgroups $M\subseteq N$ in $\mathcal N$ that the diagram
$$\xymatrix{\Pi_M\ar@{->}^{\lambda_M}[r]\ar@{->}[d]&G/\widetilde M\ar@{->}[d]\\
            \Pi_N\ar@{->}^{\lambda_N}[r]           &G/\widetilde N}$$
with vertical epimorphisms natural, commutes, 
a standard inverse limit argument (cf.  \cite[Lemma 1.1.5]{RZ:10}) shows that $G\cong \Pi$, 
where $\Pi$ is the fundamental pro-$p$ group of the same graph of groups
(cf. \hyperref[euler]{\textup{Figure~\ref{euler}}})
with initial vertex group $G_e$ and, for every $x\in X$ an edge groups $G_e$ connected to both ends
of the same vertex group $G_v$ whenever there is $v\in V$ such that $xG_ex^{-1}\subseteq G_v$. 

Now certainly $\Pi$ splits over a finite edge group as an amalgamated free pro-$p$ product or HNN-extension,
depending whether removing this edge we have one connected component or two.
\end{proof}

%\newpage
\begin{thm} %3.7
\label{t:fund}
A finitely generated pro-$p$ group $G$ acting
virtually freely on a pro-$p$ tree $T$
is isomorphic to the
fundamental pro-$p$ group 
of a finite connected graph of finite $p$-groups
whose edge and vertex groups are isomorphic to
the stabilizers of some edges and vertices of $T$.
\end{thm}

\begin{proof}
Let $F$ be a maximal normal free pro-$p$
subgroup $F$ of $G$ of smallest rank.
The proof is by induction on $\operatorname{rank}(F)$.
If $\operatorname{rank}(F)=0$, that is $G$ is finite,
then we take as graph of groups the single vertex group $G$;
the result follows from
\hyperref[t:trees1]{Theorem~\ref{t:trees1}(d)}.
In the general case, we apply
\hyperref[t:treeacting]{Theorem~\ref{t:treeacting}}
to split $G$ over the stabilizer $K$ of an edge of $T$
either as an amalgamated free pro-$p$ product
$G=G_1\amalg_K G_2$ or
as a pro-$p$ $\operatorname{HNN}$-extension
$G=\operatorname{HNN}(G_1,K,t)$.
Note that, 
replacing $G_2$ by $G_1$ if necessary,
we may assume that $K$ is contained,
up to conjugation, in $G_1$.

Now each amalgamated free pro-$p$ factor, 
or the base group,
satisfies the induction hypothesis.
Indeed, to prove it say for
$G_1$, we apply
\hyperref[p:freeedgeaction]{Proposition~\ref{p:freeedgeaction}} to
$F$ to deduce that $G_1\cap F$ is a free pro-$p$ factor of $F$.
Let us prove that  $G_1\cap F\neq F$. 
If this were not the case then \cite[Theorem~3.12]{RZ:00a} implies that $F$ would act 
trivially on the standard graph corresponding to
the presentation of $G$ as an amalgamated free pro-$p$ group or pro-$p$ $\operatorname{HNN}$-extension.
Hence $F$ would fix the edge consisting of the coset $1K$ and so $F\subseteq K$, a contradiction as
$F$ is infinite and $K$ is finite. 
Thus the rank of $G_1\cap F$ is less than the rank of $F$. 
Besides, if $G_1\cap F$ were not maximal in $G_1$, then the rank
of the maximal open free subgroup in $G_1$ could not be larger than the rank of $G_1\cap F$ 
in light of the Schreier formula
(by \hyperref[p:freeedgeaction]{Proposition~\ref{p:freeedgeaction}} applied to $F$).
Therefore the induction hypothesis holds for $G_1$ and $G_2$.

So, $G_1$ and $G_2$ are fundamental pro-$p$ groups
of finite connected graphs of finite $p$-groups.
By \cite[Thm.~3.10]{ZM:89}
(see \hyperref[t:trees1]{Theorem~\ref{t:trees1}(d)}),
$K$ is conjugate to some vertex group of $G_1$ and so
we may assume that
$K$ is contained in a vertex group of $G_1$.
In the case of an amalgamated product
there exists $g_2$ in $G_2$ such that $K^{g_2}$ is contained in
a vertex group of $G_2$, so $G$ admits a decomposition
$G=G_1^{g_2}\amalg_{K^{g_2}} G_2$.
Thus in both cases $G$ becomes
the fundamental pro-$p$ group 
of a finite connected graph of finite $p$-groups.
\end{proof}

%\newpage
The obvious fact that
the order of every finite subgroup is bounded by
the index of any torsion-free subgroup,
says that the vertex and edge groups in
\hyperref[t:treeacting_intro]{Theorem \ref{t:treeacting_intro}(b)}
are bounded by the index of any
open subgroup acting freely by restriction on $T$.
We should also mention that, 
due to the finiteness of the graph in
\hyperref[t:fund]{Theorem \ref{t:fund}},
the group $G$ in
\hyperref[t:treeacting_intro]{Theorem \ref{t:treeacting_intro}}
is the pro-$p$ completion of some dense
finitely generated discrete virtually free subgroup.

Next, we easily derive
\hyperref[t-subgrouptheorem_intro]{Theorem~\ref{t-subgrouptheorem_intro}}
from
\hyperref[t:treeacting_intro]{Theorem \ref{t:treeacting_intro}}.

\begin{proof}[Proof of \protect{\upshape
\hyperref[t-subgrouptheorem_intro]{Theorem~\ref{t-subgrouptheorem_intro}}}]
The fundamental pro-$p$ group $G$
acts naturally on its standard pro-$p$ tree $T$
(\textit{cf.} \cite[Sec.~3]{ZM:89}).
Moreover, since we have a finite graph of finite groups,
there exists an open normal subgroup of $G$
that intersects all vertex groups trivially, and 
hence acts freely on $T$.
Applying
\hyperref[t:fund]{Theorem~\ref{t:fund}}
to the action restricted to $H$,
we immediately obtain the result.
\end{proof}

%\newpage
We end this section by recalling the example from \cite{HZ:10}
of an infinite countably generated pro-$p$ group
acting virtually freely on a pro-$p$ tree that
satisfies none of the conclusions of
\hyperref[t:treeacting_intro]{Theorem \ref{t:treeacting_intro}}.

\begin{example}
\label{ex:4.3} \upshape
Let $A$ and $B$ be groups of order $2$ and
$H$ 
be a pro-$2$ $\operatorname{HNN}$-extension 
with base group $A\times B$,
associated subgroups $A$ and $B$,
and stable letter $t$.
Note that $H$ admits an automorphism $\varphi$
of order $2$ that swaps $A$ and $B$ and inverts $t$.
Moreover,
the holomorph
$H\rtimes \langle\varphi\rangle$
is isomorphic to
$((A\times B)\rtimes \langle\varphi\rangle)
\amalg_{A}
(A\times\langle\varphi t\rangle)$,
the free pro-$2$ product of
the dihedral group of order $8$
and the Klein four-group
amalgamated along the cyclic group $A$.
Let
$G=\langle\operatorname{tor}(H)\rangle
\rtimes\langle\varphi\rangle$.
Since
the amalgamated free pro-$2$ product
acts virtually freely on its standard pro-$2$ tree,
the group $G$ acts virtually freely on this pro-$2$ tree by restriction.
The main result of Herfort and Zalesskii
\cite{HZ:10}
shows that $G$ does not split over a finite group as %the
fundamental pro-$2$ group of a profinite graph of finite $2$-groups.
Its proof also shows that $G$ does not decompose as
an amalgamated free pro-$2$ product or
as a pro-$2$ $\operatorname{HNN}$-extension.
\end{example}

%%%%%%%%%%%%%%%%%%%%%%%%%%%%%%
%\newpage
\phantomsection
\section{\texorpdfstring{$2$}{2}-generated subgroups}
\label{s:2-generated}
%%%%%%%%%%%%%%%%%%%%%%%%%%%%%%

This final section is devoted to %the proof of
\hyperref[t:freeorabelian_intro]{Theorem~\ref{t:freeorabelian_intro}}.
We begin presenting auxiliary results 
%(\hyperref[p:2gen]{Proposition~\ref{p:2gen}}-
%\hyperref[l:fp]{Lemma~\ref{l:fp}})
that are used to prove it.
%\hyperref[t:freeorabelian_intro]{Theorem~\ref{t:freeorabelian_intro}}.
%six of which being about inverse limits.

%\newpage
\begin{prop} %2.7
\label{p:2gen}
%\small
 Let $G$ be a $2$-generated pro-$p$ group.
 \begin{itemize}
 \item[\textup{(a)}]
 If $G$ is a free pro-$p$ product with
 procyclic amalgamation,
 then one of its amalgamated free pro-$p$ factors is procyclic.
 \item[\textup{(b)}]
 If $G$ is a proper pro-$p$ $\operatorname{HNN}$-extension
 with procyclic associated subgroups,
 then its base subgroup is at most $2$-generated.
 \item[\textup{(c)}]
 If $G$ is the fundamental pro-$p$ group of a finite non-trivial tree,
 of finite $p$-groups such that all edge groups are cyclic,
 then either %$|G|< \infty$ or 
 $G=K\amalg_C R$ with $K$ finite cyclic and $R$ finite,
 or $G=K\amalg_D M\amalg_E N$,
 with $K$ and $N$ finite cyclic and $M\subseteq \Phi(G)$.
 \end{itemize}
\end{prop}

\begin{proof}
\noindent (a) Let $G=A\amalg_{C} B$, and
let ``bar'' indicate passing to the pro-$p$ Frattini quotients
$A/\Phi(A)$, $B/\Phi(B)$, and $C/\Phi(C)$.
We have an obvious epimorphism
from $G$ to the induced pushout 
$\overline{A}\amalg_{\overline{C}} \overline{B}$,
denoted by $P$.
Let $n\!:=d(A)+d(B)$.
Since $C$ is procyclic, the image $M$ of the
kernel of the canonical map
$\overline{A}\amalg \overline{B} \to P$
via the cartesian map
$\overline{A}\amalg \overline{B}\to \overline{A}\times \overline{B}$
is also procyclic.
The latter map induces an epimorphism from $P$
to the at least $(n-1)$-generated elementary abelian pro-$p$ group
$(\overline{A}\times \overline{B})/M$.
Therefore, $n-1\le d(G)$ and the result follows.

\noindent (b) Let $G=\operatorname{HNN}(H,C,f,t)$
with $C={\langle c\rangle}$, and
denote by ``bar'' passing to pro-$p$ Frattini quotients.
%If $d(H)\ge 3$ then $d(G)\ge 3$
%as can be seen by 
From the obvious epimorphism
$G\to (\overline{H}\times \overline{\langle t\rangle})/
%{\langle tct^{-1}f(c)^ {-1}\rangle}$.
{\langle 
\overline{t}^{-1}\overline{c}\overline{t}(\overline{f(c)})^{-1}
\rangle}$
it follows that
$d(H)\le d(G)$.
%Thus $d(H)\le 2$

\noindent (c) Let $G=\Pi_1(\mathcal{G},\Gamma)$,
with finite vertex groups $\mathcal{G}(v)$
and cyclic edge groups $\mathcal{G}(e)$.
We claim that $\Gamma$ has at most $3$ vertices.
Indeed, 
%By assumption $|V(\Gamma)|\ge2$, 
%and therefore it has an edge $e$. 
splitting $G$ over an edge $e$ of $\Gamma$,
we may and do assume that
$\mathcal{G}(d_0(e))$ is procyclic by item (a); 
hence $d_0(e)$ is a {\it pending vertex} of $\Gamma$
(\textit{i.e.}, $e$ is the unique edge incident to the vertex).
Suppose now that $\Gamma$ has at least $3$ vertices, 
and let $a$ be an arbitrary edge of $\Gamma\!-\!\{e\}$
having initial or terminal vertex $d_1(e)$;
without loss of generality, suppose that $d_0(a)=d_1(e)$.
Then $d_1(a)$ is a pending vertex with
procyclic vertex group $\mathcal{G}(d_1(a))$; 
for, otherwise, by splitting $G$ over the edge $a$ 
we would obtain that $d(G)>2$, a contradiction.
Now, if we have a number $r\ge 2$ of edges with initial or
terminal vertex $d_1(e)$ then
it follows from the pro-$p$ presentation of $G$ 
%EXPLAIN? add presentation???
that it has a free pro-$p$ abelian group 
%$\mathbb{Z}_p^{r}$
%$\mathbb{Z}_p^{\,r}$
%${\mathbb{Z}_p}^{r}$
${\mathbb{Z}_p}^{\!\!\!r}$
as a quotient;
this implies $r=2$, whence $|V(\Gamma)|\le 3$.

If $|V(\Gamma)|=2$ then $G=K\amalg_D M$
with $K$ and $M$ finite,
and, by item (a), we can assume that $K$ is cyclic.

Finally suppose that $|V(\Gamma)|=3$.
Then $G=K\amalg_D M\amalg_E N$
with $K$, $M$, and $N$ finite, and $D$ and $E$ cyclic.
Since the decomposition of $G$ is proper, we have
%By the properness of our decomposition we have
$d(K\amalg_D M)=d(M\amalg_E N)=2$ and, 
making use of item (a),
we conclude that $K$ and $N$ must both be cyclic.
Since $d(G)=2$ then $M\subseteq \Phi(G)$ follows.
\end{proof}

%\newpage
\begin{lemma} %2.8
\label{l:2gen_hnn}
%\small
Let $G=\operatorname{HNN}(H,A,t)$ be a proper
pro-$p$ $\operatorname{HNN}$-extension.
Suppose that $G$ is a $2$-generated pro-$p$ group and 
$A$ is procyclic.
\begin{itemize}
 \item[\textup{(a)}]
 If $H$ is a free pro-$p$ group of rank $2$, then so is $G$.
 \item[\textup{(b)}]
 Assume that the centralizer $C_G(A)$ is abelian.
 If $H$ is a free abelian pro-$p$ group, then so is $G$.
\end{itemize}
\end{lemma}
%cited in 4.2

\begin{proof}
\noindent (a) 
Since $d(G)\le 2$,
either $A$ or $A^t$ is not contained in the Frattini $\Phi(H)$
(see the proof of 
\hyperref[p:2gen]{Proposition~\ref{p:2gen}(b)}); 
thus,
the procyclicity of $A$ gives that 
either $A$ or $A^t$ is a free factor of $H$
(\textit{cf.} \cite[Lemma~9.1.18]{RZ:10}).
Without loss of generality suppose $H=A\amalg E$.
Let 
%$d$ and $f(d)$ be generators of $D$ and $D^t$, respectively.
$a$ be a generator of $A$ and
$f(a)$ be the corresponding generator of $A^t$.
Then $G\cong(A\amalg E\amalg\langle t\rangle)/
\langle a^tf(a)^{-1}\rangle^{A\amalg E\amalg\langle t\rangle}$
and $a^tf(a)^{-1}\not\in\Phi(A\amalg E\amalg\langle t\rangle)$ 
because $(A\amalg E)\cap (A\amalg E)^t=\{1\}$ 
(see 
\hyperref[t:afpproperties]{Theorem~\ref{t:afpproperties}(b)}).
So, $G$ is a free pro-$p$ group of rank 2.

\noindent (b)
Suppose that $G$ is not a free abelian pro-$p$ group.
Note that $G$ contains the non-abelian subgroup
$\langle H,H^t\rangle$, %K\ne D
but 
$C_G(A)$ is abelian and also contains
$\langle H,H^t\rangle$;
a contradiction.
\end{proof}
%check!

%\newpage
We shall use the following general lemma,
whose profinite version is essentially
\cite[Prop.~2.2(ii)]{BZ:11}, to prove
\hyperref[p:decomphnn]{Proposition~\ref{p:decomphnn}};
we give a proof of it for the convenience of the reader.

%\newpage
\begin{lemma} %2.9
\label{l:hnn_conjug}
%\small
Let $G=\operatorname{HNN}(H,A,t)$ be a proper
pro-$p$ $\operatorname{HNN}$-extension.
If $A$ and $A^t$ are not conjugate in $H$,
then
$N_G(A)=
N_{H\phantom{^{t^{-1}}}\!\!\!\!\!\!\!\!\!}(A)
\amalg_A N_{H^{t^{-1}}}(A)$.
\end{lemma}
%cited in 2.10

\begin{proof}
Let $T$ be the standard pro-$p$ tree on which $G$ acts
(\textit{cf.} \cite[Sec.~4]{RZ:00a}).
By restriction,
the normalizer $N_G(A)$ acts on
the pro-$p$ subtree $T^A$.
We claim that the quotient graph
$N_G(A)\backslash T^A$ is either a loop or a segment.
In fact, if $gA$ is an edge of $T^A$ then $g\in N_G(A)$
(because closed subsemigroups in compact groups are subgroups);
that is, we have a unique edge in the quotient.
%So, $T^A/N_G(A)$ is either a loop or a segment.
Now, the hypothesis of $A$ and $A^t$ %$B=t^{-1}At$
being non-conjugate in $H$
is equivalent to $N_G(A)1H\ne N_G(A)tH$;
hence $N_G(A)\backslash T^A$ is a segment.
From the pro-$p$ version of \cite[Prop.~4.4]{ZM:90},
%Suppose {H} acts on a simply connected graph {Sigma}, such that the quotient {Gamma = H \ Sigma} is finite. Then {H} is the fundamental group of the graph of groups {(\mathcal{G}, Gamma)}, where, for each {m} in {Gamma}, {\mathcal{G} (m)} is isomorphic to the stabilizer in {H} of some pre-image of {m} in {Sigma}.
the lemma follows.
\end{proof}

{\it For the next three more technical propositions, 
directed posets $I$ of inverse systems
are assumed to be isomorphic to $\mathbb N$.}

%\newpage
\begin{prop} %4.4
\label{p:decomphnn}
%\small
Let $G$ be the inverse limit of a surjective inverse system
%Let $G$ be the strict inverse limit
$\{G_i,\varphi_{ij},I\}$ of pro-$p$ groups.
Suppose that each
$G_i$
is equal to a proper pro-$p$ $\operatorname{HNN}$-extension
$\operatorname{HNN}(H_i,A_i,B_i,t_i)$
with $H_i$ finite. %and $\varphi_{ij}(H_i)\cong H_j$.
We have:
%\newpage
\begin{itemize}
\item[\textup{(a)}]
There exists an inverse system of groups
$\{H_i',\varphi_{ij},J\}$
where $J$ is a cofinal subset of $I$,
and each $H_i'$ is a conjugate to a subgroup of $H_i$
by an element of $G_i$.
\item[\textup{(b)}]
If $\varphi_{ij}(H_i)\cong H_j$,
then there exists a cofinal subset $J$ of $I$ such that, 
either there is an inverse system of groups
$\{A_i'',\varphi_{ij},J\}$ with
each $A_i''$ conjugate to $A_i$
by an element of $G_i$, 
or there is an inverse system of groups
$\{B_i'',\varphi_{ij},J\}$ with
each $B_i''$ conjugate to $B_i$
by an element of $G_i$.
\item[\textup{(c)}]
If, in addition to the assumptions in (b), 
%$\varphi_{ij}(H_i)\cong H_j$ and moreover 
$\varphi_{ij}(A_i)\cong A_j$,
then  $G=\operatorname{HNN}(H,A,B,t)$ with
 $H\!:=\varprojlim H_i'$,
$A\!:=\varprojlim A_i''$ and $B\!:=\varprojlim B_i''$.
%\look{cutting}
%where each $B_i''$ is conjugate to $B_i$ 
%for some element in $G_i$.
\end{itemize}
\end{prop}

\begin{proof} \
Fix $l$ and $k$ in $I$ with $k\le l$.
By 
\hyperref[t:hnnproperties]{Theorem~\ref{t:hnnproperties}(a)}
there is an element $g_k$ in $G_k$ with
$\varphi_{lk}(H_l)\subseteq H_k^{g_k^{-1}}$.
Conjugating $H_l$, $A_l$, $B_l$ and $t_l$
by an element $g_l$ in $\varphi_{lk}^{-1}(g_k)$,
we do not change $G_l$, and 
we obtain $H_l'$, $A_l'$, $B_l'$ and $t_l'$ such that
$\varphi_{lk}(H_l')\subseteq H_k'$.

\medskip 

\noindent (a) \ 
The cofinal subset $J$ of $I$ 
is inductively constructed.

\medskip

\noindent (b) \ 
Since $A_l = H_l \cap H_l^{t_l^{-1}}$ and
$\varphi_{lk}$ is surjective, 
by 
\hyperref[t:hnnproperties]{Theorem~\ref{t:hnnproperties}(b)}
we have that
$\varphi_{lk}(A_l')$ is,
up to conjugation by an element of $h_k^{-1}$ of $H_k'$,
contained in $A_k'$ or $B_k'$.
By the same argument $\varphi_{lk}(B_l')$ is,
up to conjugation by an element of $H_k'$,
contained in $A_k'$ or $B_k'$.

Suppose that for infinitely many indices $i$ and $j$ of $I$
we have that $A_i'$ are, up to conjugation, sent to $A_j'$.
Conjugating $H_i'$, $A_i'$, $B_i'$ and $t_i'$
by an element $h_i$ of $\varphi_{ij}^{-1}(h_j)\cap H_i'$
(note that such $h_i$ exists because $\varphi_{ij}(H_i)\cong H_j$), 
we obtain $H_i'$, $A_i''$, $B_i''$ and $t_i''$ such that
$\varphi_{ij}(A_i'')$ is contained in $A_j'$;
this does not change the group $G_i$.
Similarly as in (a), inductively, we obtain
an inverse system $\{ A_i'', \varphi_{ij},  J \}$ with $J$ cofinal to $I$. 

Otherwise, for almost all $i>j$, up to conjugation in $H_j$
the subgroup $A_i$ is sent to $B_j'$.
Then inductively we obtain
an inverse system $\{ B_i'', \varphi_{ij},  J \}$ with $J$ cofinal to $I$. 

%\newpage
\medskip

\noindent (c) \ 
Passing to a cofinal subset of $I$,
%an infinite subset of J, or the complement of a finite subset of J
for each $i$ we have two cases:
(i) $A_i''$ and $B_i''$ are conjugate in $H_i'$;
(ii) $A_i''$ and $B_i''$ are not conjugate in $H_i'$.

In case (i), after conjugation not changing $G_j$, we may
suppose that $A_i''$ and $B_i''$ coincide.
Thus, we obtain coinciding inverse systems
$\{A_i'',\varphi_{ij},J\}$ and $\{B_i'',\varphi_{ij},J\}$.

In case (ii),
we may assume that each $\varphi_{ij}(A_i'')$
and $\varphi_{ij}(B_i'')$
coincides with $A_j''$ or $B_j''$,
since $\varphi_{ij}(A_i)\cong A_j$.
Now, it cannot happen that both $A_i''$ and $B_i''$ are sent
to the same associate subgroup of $H_j'$.
Indeed, if say $\varphi_{ij}(A_i'')=A_j''=\varphi_{ij}(B_i'')$,
then $\varphi_{ij}(t_i'')$ would normalize $A_j''$
and thus $\varphi_{ij}(t_i'')$ would be contained 
in the proper subgroup
$\langle H_k',{H_k'}^{t_k''^{-1}}\rangle$,  
by 
\hyperref[l:hnn_conjug]{Lemma~\ref{l:hnn_conjug}};
this  contradicts to $\varphi_{ij}$ being surjective.
Therefore,
either $\varphi_{ij}(A_i'')\subseteq A_j''$
and $\varphi_{ij}(B_i'')\subseteq B_j''$,
or $\varphi_{ij}(A_i'')\subseteq B_j''$
and $\varphi_{ij}(B_i'')\subseteq A_j''$.
Passing to a cofinal subset,
we obtain inverse systems
$\{A_i'',\varphi_{ij},J\}$ and $\{B_i'',\varphi_{ij},J\}$.

Now, let $H\!:=\varprojlim H_i'$,
$A\!:=\varprojlim A_i''$, $B\!:=\varprojlim B_i''$,
%$B\!:=\varprojlim {A_i'}^{t_i'}$
and let $\varphi_i\colon G\to G_i$ be the projections.
For each $i$ in $I$ let us consider the subset
\[
X_i\!:=\{\tau_i\in G\mid\varphi_i(A)^{\varphi_i(\tau_i)}\!=\!\varphi_i(B)
{ \, \ \text{and} \, \ }
G_i={\langle \varphi_i(H), \varphi_i(\tau_i)\rangle} \} \ . 
\]
Clearly every $X_i$ is a non-empty compact set, and
since $X_{i+1}\subseteq X_i$, %for all $i\in I$,
there exists an element $t$ in $\bigcap_i X_i$ such that $B=A^t$.

The existence of the desired isomorphism from
$\operatorname{HNN}(H,A,B,t)$ onto $G$ follows now from
the universal property of 
pro-$p$ $\operatorname{HNN}$-extensions.
\end{proof}

%\newpage
\begin{prop} %4.5
\label{p:decompfp}
%\small
Let $G$ be the inverse limit of a surjective inverse system
%Let $G$ be the strict inverse limit 
$\{G_i,\varphi_{ij},I\}$ of pro-$p$ groups.
Suppose that each $G_i$ is equal to a free pro-$p$ product
$A_i\amalg B_i$ with $A_i$ finite cyclic and $B_i$ procyclic.
We have:
\begin{itemize}
\item[\textup{(a)}]
If some $B_i$ is infinite, then 
there exists an inverse system $\{A_{i}',\varphi_{ij},J\}$ 
where $J$ is a cofinal subset of $I$,
and each $A_i'$ is a conjugate of $A_i$
by an element of $G_i$, such that
$G\cong \left(\varprojlim A_{i}'\right) \amalg \mathbb{Z}_p$.
\item[\textup{(b)}]
If each $B_i$ is finite, then
there exist inverse systems
$\{A_{i}',\varphi_{ij},J\}$ and $\{B_{i}',\varphi_{ij},J\}$,
where $J$ is a cofinal subset of $I$,
and each $A_i'$ (resp. $B_i'$) is a conjugate of $A_i$ (resp. $B_i$)
by an element of $G_i$, such that
 $G\cong \left(\varprojlim A_{i}'\right)
 \amalg \left(\varprojlim B_{i}'\right)$.
\end{itemize}
\end{prop}

\begin{proof} 
\noindent (a)
Let $i_0\in I$ such that $B_{i_0}\cong {\mathbb Z}_p$,
we have $B_i=\langle t_i\rangle\cong {\mathbb Z}_p$
for each $i$ with $i_0\le i$.
By 
\hyperref[t:afpproperties]{Theorem~\ref{t:afpproperties}(a)},
$A_i$ is mapped by $\varphi_{ij}$ to a conjugate of $A_j$.
And in fact, surjectively; 
for, otherwise,
the induced homomorphism between the cartesian products
$A_i\times B_i \to A_j\times B_j$
would not be surjective.
To obtain the desired result,
we apply
\hyperref[p:decomphnn]{Proposition~\ref{p:decomphnn}(c)}
with $G_i= \operatorname{HNN}(A_{i},\{1\},t_i)$.

\medskip
%\newpage
\noindent(b)
Since each $B_{i}$ is finite
and each $\varphi_{ij}$ is surjective,
from 
\hyperref[t:afpproperties]{Theorem~\ref{t:afpproperties}(a)},
we obtain that distinct free factors of $G_i$ are,
up to conjugation, mapped to distinct free factors of $G_j$.
So, there exists a cofinal subset $J$ of $I$ such that
for every $i$, $j$ in $J$ we have
$\varphi_{ij}(A_{i})= A_{j}^{x_{j}}$
and 
$\varphi_{ij}(B_{i})= B_{j}^{y_{j}}$
for certain $x_{j},y_{j}$ in $G_j$.
Then, after conjugations of free factors,
we inductively obtain
the desired inverse systems
$\{A_{i}',\varphi_{ij},I\}$ and
$\{B_{i}',\varphi_{ij},I\}$.
The result follows from \cite[Lemma~9.1.5]{RZ:10}.
\end{proof}

%\newpage
\begin{prop} %4.6
\label{p:treeprod}
%\small
Let $G$ be the inverse limit of a surjective inverse system
%Let $G$ be the strict inverse limit 
$\{G_i,\varphi_{ij},I\}$ of  pro-$p$ groups $G_i$.
Suppose $G_i$ is equal to an amalgamated free pro-$p$ product
$G_i=K_i\amalg_{D_i} R_i$ with $K_i$ finite cyclic and $R_i$ finite
or $G_i=K_i\amalg_{D_i}M_i\amalg_{E_i}N_i$,
with $K_i$ and $N_i$ finite cyclic and $M_i\subseteq \Phi(G_i)$.
Then, %passing to a cofinal subset of $I$, if necessary,
there exist inverse systems
$\{K_i',\varphi_{ij},J\}$ and $\{D_i'',\varphi_{ij},J\}$,
where $J$ is a cofinal subset of $I$,
such that $D_i''\subseteq K_i'$,
$\varphi_{ij}(K_i')= K_j'$ and $\varphi_{ij}(D_i'')\subseteq D_j''$
where each $K_i'$ (resp. $D_i''$) is a conjugate of $K_i$ (resp. $D_i$)
 by an element of $G_i$.
\end{prop}

\begin{proof}
Using 
\hyperref[t:afpproperties]{Theorem~\ref{t:afpproperties}(a)},
if $G_i$ decomposes in the first manner, then
$\varphi_{ij}$ sends amalgamated free pro-$p$ factors to
amalgamated free pro-$p$ factors up to conjugation;
if $G_i$ decomposes in the second manner, then 
$\varphi_{ij}$ sends cyclic amalgamated free pro-$p$ factors to 
cyclic amalgamated free pro-$p$ factors up to conjugation.
So, in both cases, we may pass to a cofinal subset $J$ of $I$ such that
for all $i$ and $j$ in $J$ with $i\ge j$ we have
$\varphi_{ij}(K_i)\subseteq K_j^{g_j}$, for some $g_j$ in $G_j$.
Then,
since $K_j$ is cyclic, we obtain $\varphi_{ij}(K_i)=K_j^{g_j}$
(in fact, otherwise $\varphi_{ij}(K_i)^{G_j}\neq K_j^{G_j}$ 
would contradict the surjectivity of $\varphi_{ij}$).
Now, selecting any $g_i$ in$\varphi_{ij}^{-1}(g_j)$
and letting $K_i'\!:=K_i^{g_i^{-1}}$,
inductively we obtain the desired inverse system
$\{K_i',\varphi_{ij},J\}$.
Next, letting $D_i'\!:=D_i^{g_i^{-1}}$ we have
$D_i'\subseteq K_i'\cap M_i^{g_i}$; 
then, by
\hyperref[t:afpproperties]{Theorem~\ref{t:afpproperties}(b)},
$\varphi_{ij}(D_i')\subseteq K_j\cap \varphi_{ij}(M_i)
\subseteq {D_j'}^{b_j}$,
for some $b_j$ in $K_j'$.
Choosing any $b_i$ in $\varphi_{ij}^{-1}(b_j)\cap K_i'$
and letting $D_i''\!:={D_i'}^{b_i^{-1}}$ we obtain
the other inverse system $\{D_i'',\varphi_{ij},I\}$.
\end{proof}

%\newpage
The following three simple general lemmas
will also be used in the last section of the paper.

\begin{lemma} %2.13
\label{l:invsys_d-gen}
%\small
Let $G$ be the inverse limit of
an inverse system $\{G_i,\varphi_{ij},I\}$ of
pro-$p$ groups.
Suppose that there is a constant $d$ with $d(G_i)=d$ for all $i$ in $I$.
If $d(G)=d$, then 
there exists $k$ in $I$ such that
$\varphi_{ik}$ is surjective for each $i$.
In particular,
the projections $G\to G_i$ are surjective whenever $i\ge k$.
\end{lemma}

\begin{proof}
For each $i$ in $I$, let $\varphi_i\colon G\to G_i$ be the
projection.
By \textit{reductio ad absurdum},
suppose that
for each $j$ in $I$ there were some non-surjective $\varphi_{ij}$.
Each such $\varphi_{ij}$ induces a map
$G_i/\Phi(G_i)\to G_j/\Phi(G_j)$
which would also be non-surjective
(\textit{cf.} \cite[Prop.~7.7.2]{RZ:10});
in particular, $\varphi_{j}(G)\Phi(G_j)/\Phi(G_j)$ 
would be a proper subgroup of $G_j/\Phi(G_j)$.
Since $G/\Phi(G)\cong \varprojlim_j \varphi_{j}(G)\Phi(G_j)/\Phi(G_j)$,
the quotient $G/\Phi(G)$,
and hence $G$, could be generated by $d-1$ elements;
a contradiction.
For the last assertion of the lemma see
\cite[Prop.~1.1.10]{RZ:10}.
\end{proof}

%\newpage
\begin{lemma}%2.14
\label{l:invsys_ncl}
%\small
Let $G$ be the inverse limit of
an inverse system $\{G_i,\varphi_{ij},I\}$ of profinite groups.
Suppose we have closed subgroups $H_i$ of $G_i$
such that $\varphi_{ij}(H_i)\subseteq H_j$ whenever $i\ge j$.
%Suppose that there is a constant $d$ with $d(G_i)=d$ for all $i\in I$.
Then, for the induced inverse limit $H=\varprojlim H_i$,
%\subseteq G
we have the following equality among normal closures
$H^G=\varprojlim H_i^{G_i}$. 
\end{lemma}

\begin{proof}
Let $P$ denote the cartesian product $\prod_{i\in I}G_i$, and
let $\mathcal{F}$ be the set of all {finite} subsets of $I$.
For any $J$ in $\mathcal F$,
the canonical embedding of $G$ into $P$ and
the projections $\pi_i\colon P\to G_i$
give rise to subgroups 
\[
G_J:=\{g\in P\mid (\forall i,j\in J)\ i\le j \Rightarrow \varphi_{ji}(\pi_j(g))=\pi_i(g)\}
\]
and
\[
H_J:=\{h\in P\mid (\forall i,j\in J)\ \pi_i(h)\in H_i \hbox{ and } i\le j 
\Rightarrow \varphi_{ji}(\pi_j(h))=\pi_i(h)\} \, .
\]
Note that %$G=\bigcap_{J\in\mathcal{F}}G_J$ and
%$H=\bigcap_{J\in\mathcal{F}}H_J$, and clearly
the families $\{G_J\mid J\in \mathcal{F}\}$ and
$\{H_J\mid J\in \mathcal{F}\}$ are
filtered from below
and have intersections $G$ and $H$, respectively.
Now, let 
$\mathcal{N}$ be the set of all clopen normal subgroups of $P$.
For any $N$ in $\mathcal N$,
we have two {\em finite} 
families of subgroups of $P$ filtered from below
 $\{G_JN\mid J\in \mathcal{F}\}$ and
$\{H_JN\mid J\in \mathcal{F}\}$, and so
certainly in $P/N$ we have 
$\bigcap_{J\in \mathcal{F}}(H_JN)^{G_JN}=
(\bigcap_{J\in \mathcal{F}}H_JN)^{\bigcap_{J\in\mathcal{F}}G_JN}$.
The latter equality reads 
\[
\bigcap_{J\in\mathcal{F}}(H_J^{G_J}N)=H^{G}N \, .
\]
Now it follows from
$\bigcap_{N\in\mathcal{N}}\bigcap_{J\in\mathcal{F}}(H_J^{G_J}N)=
\bigcap_{J\in\mathcal{F}}\bigcap_{N\in\mathcal{N}}(H_J^{G_J}N)=
\bigcap_{J\in\mathcal{F}}H_J^{G_J}$
and $\bigcap_{N\in\mathcal{N}}H^{G}N=H^G$ 
that $H^G=\bigcap_{J\in\mathcal{F}}H_J^{G_J}=
\bigcap_{J\in\mathcal{F}}\bigcap_{j\in J}\pi_j^{-1}(H_j^{G_j})$.
It remains to observe that the latter intersection coincides with 
$\varprojlim_iH_i^{G_i}$.
\end{proof}
%check!

%be careful with right action!!!
%\newpage
\begin{lemma}%2.15
\label{l:fp}
%\small
Let $G$ be a pro-$p$ group acting on 
a compact space $\Omega$.
Let $(H_n)_{n\ge 1}$ be a sequence of
normal subgroups of $G$ satisfying
$H_{n+1}\subseteq H_n$ for each $n\ge 1$,
and $\bigcap_{n\ge 1} H_n=\{1\}$.
Define $G_n=G/H_n$
and $\Omega_n=H_n\backslash \Omega$.
Suppose that there exist subgroups $S_n$ of $G_n$ 
such that $\varphi_{nm}(S_n)\subseteq S_m$
where $\varphi_{nm}$ are the canonical maps.
%and let $S=\varprojlim S_n$ be the inverse limit.
Let  $S=\varprojlim S_n$.
If %, with respect to the obvious actions,
$\Omega_n^{S_n}\neq\emptyset$ for 
%infinitely many integers $n$
each $n\ge 1$,
then $\Omega^S\neq\emptyset$.
\end{lemma}

\begin{proof}
Denote by
$\varphi_n\colon G\to G_n$ and
$\pi_n\colon \Omega\to \Omega_n$
the canonical projections.
Then
$\Omega_n^{\varphi_n(S)}\supseteq \Omega_n^{S_n}\neq\emptyset$.
%Therefore
%$Y_n:=\pi_n^{-1}(\Omega_n^{\varphi_n(S)})\neq\emptyset$.
Denoting the non-empty set
$\pi_n^{-1}(\Omega_n^{\varphi_n(S)})$ by $Y_n$,
we have $Y_n=\{x\in \Omega\mid Sx\subseteq H_nx\}$,
so that $Y_{n+1}\subseteq Y_n$; 
by the compactness of $\Omega$ it follows that
$\emptyset\neq \bigcap_{n\ge 1} Y_n\subseteq \Omega^S$.
\end{proof}

Now, to prove 
\hyperref[t:freeorabelian_intro]{Theorem~\ref{t:freeorabelian_intro}}
\textit{we henceforth assume that 
$G:=A\amalg_{C} B$ is a free pro-$p$ product of
$A$ and $B$ with procyclic amalgamating subgroup $C$
satisfying the succeeding conditions:}

\begin{itemize}
\item[\textbf{C1.}]
the centralizer in $G$ of %each non-trivial closed subgroup of 
$C$ is a free abelian pro-$p$ group and contains $C$ as a direct factor.

\item[\textbf{C2.}]
each $2$-generated  subgroup of $A$
and each $2$-generated subgroup of $B$
is either a free pro-$p$ group or a free abelian pro-$p$ group.
\end{itemize}

It is a consequence of the next simple lemma that,
in the presence of \textbf{C2},
condition \textbf{C1} is equivalent to the following condition:

\begin{itemize}
 \item[\textbf{C1$'$.}]
  the centralizer in $G$ of each non-trivial subgroup of $C$ is
  a free abelian pro-$p$ group and contains $C$ as a direct factor.
\end{itemize}

\begin{lemma}
\label{l:cent}
%\small
For each non-trivial subgroup $D$ of $C$ we have
$N_G(D)=C_G(D)=C_G(C)$.
\end{lemma}

\begin{proof}
We claim that $N_A(D)=C_A(D)$ and $N_B(D)=C_B(D)$ hold.
Indeed, 
first, suppose that $N_A(D)\ne C_A(D)$.
Then, there exists an element $x$ in $N_A(D) \!-\! C_A(D)$
making the $2$-generated subgroup
${\langle x,D\rangle}$ 
metabelian but non-abelian;
a contradiction with \textbf{C2}.
Second, if there exists an element $y$
in $C_A(D) \!-\! C_A(C)$, 
then $\langle y,C\rangle$
is a non-abelian free pro-$p$ group,
by \textbf{C2}.
Therefore, $\langle y,C\rangle=\langle y\rangle\amalg C$.
Since $y\in C_A(D)$,  
from
\hyperref[t:afpproperties]{Theorem~\ref{t:afpproperties}(b)}
we obtain $D\subseteq C\cap C^y=\{1\}$;
a contradiction.
Our initial claim is proved.

Finally,
by the pro-$p$ version of \cite[Cor.~2.7(ii)]{RZ:96},
we have
\[
N_{G}(D)=N_{A}(D)\amalg_{C}N_{B}(D)\, .
\]
So
$N_G(D)={\langle N_{A}(D),N_{B}(D)\rangle}
={\langle C_{A}(C),C_{B}(C)\rangle}\subseteq C_G(C)$,
as desired.
\end{proof}

Using 
\hyperref[t:treeacting_intro]{Theorem~\ref{t:treeacting_intro}}
we now prove
\hyperref[t:freeorabelian_intro]{Theorem~\ref{t:freeorabelian_intro}}.

\begin{proof}[Proof of \protect{\upshape
\hyperref[t:freeorabelian_intro]{Theorem~\ref{t:freeorabelian_intro}}}]
Let $T$ be the standard pro-$p$ tree on which $G$ acts
(\textit{cf.} \cite[Sec.~4]{RZ:00a})
and let $L$ be a $2$-generated subgroup of $G$.
It follows from the definition of $T$ that if
$L$ stabilizes a vertex of $T$, then $L$ is up to conjugation
in one of the amalgamated free factors of $G$;
hence $L$ is either free pro-$p$ or
free abelian pro-$p$, by hypothesis~\textbf{C2}.
So let us assume that $L$ fixes no vertex of $T$.

Since $L$ is finitely generated, we have 
$L\cong \varprojlim L/U_n$ where
$\{U_n\mid n\in{\mathbb N}\}$
is a sequence of open normal subgroups of $L$ with 
$U_{n+1}\subseteq U_n$ for each $n\ge 1$,
and $\bigcap U_n =\{1\}$.
Recall our notation $\widetilde{U_n}$
for the closed subgroup of $U_n$ generated by all vertex
stabilizers with respect to the action of $U_n$ on $T$.

Defining $L_n\!:=L/\widetilde{U_n}$
we have that each $L_n$ acts virtually freely 
on the pro-$p$ tree $\widetilde{U_n}\backslash T$.
%and furthermore, we may and will henceforth consider
%an infinite set $I$ of integers $n$ such that each $L_n$ is infinite.
Indeed, 
the quotient graph $\widetilde{U_n}\backslash T$ is a pro-$p$ tree
by 
\hyperref[t:trees1]{Theorem~\ref{t:trees1}(a)},
and each $U_n/\widetilde{U_n}$ is a free pro-$p$ group,
by
\hyperref[t:trees1]{Theorem~\ref{t:trees1}(b)}.
Furthermore, if $U_n=\widetilde{U_n}$ for almost every $n$,
then $L_n$ would almost always be a finite group acting on $\widetilde{U_n}\backslash T$;
thus, by
\hyperref[t:trees1]{Theorem \ref{t:trees1}(d)},
we could apply
\hyperref[l:fp]{Lemma~\ref{l:fp}}
with $\Omega\!:=V(T)$ and $S_n\!:=L_n$,
to obtain a vertex of $T$ fixed by %$L$.
$L\cong \varprojlim \{L_n, \varphi_{nm}, I\}$
where each $\varphi_{nm}$ is the canonical map.
This contradicts the assumption that $L$ does {\em not} fix any vertex.
Thus we must have $U_n\neq \widetilde{U_n}$, 
\textit{i.e.}, $L_n$ is infinite,
for almost every $n$.

In virtue of
\hyperref[t:treeacting_intro]{Theorem~\ref{t:treeacting_intro}(b)},
we have that each
$L_n$ %=\Pi_1(\mathcal{L}_n,\Gamma_n)$
is the fundamental pro-$p$ group of
a finite connected graph $\Gamma_n$
of finite $p$-groups whose edge and vertex groups
are stabilizers of certain edges and vertices of
$\widetilde{U_n}\backslash T$.

Now, since $L/\widetilde{L}$ is a free pro-$p$ group
of rank at most $2$,
we need to only examine the cases
$L=\widetilde{L}$ and $L/\widetilde{L}\cong {\mathbb Z}_p$;
in the remaining case, when $d(L/\widetilde L)=2$,
$L$ is itself free pro-$p$ of rank 2 -- by the Hopfian property
(\textit{cf.} \cite[Prop.~2.5.2]{RZ:10}).
We can assume that $\widetilde{L}\neq \{1\}$, otherwise
there is nothing to prove.

%\newpage
\medskip
\noindent Case 1.
\textit{$L=\widetilde{L}$.}
\medskip

Note that each finite connected graph $\Gamma_n$ is a tree.
Otherwise, there is an edge $e_n$ in $\Gamma_n$ such that
$L_n=\operatorname{HNN}(P_n,G(e_n),t_n)$ for $G(e_n)$ finite.
But then there is a homomorphism from $L_n$ onto $\mathbb{Z}_p$
contradicting 
$\widetilde{L}/\widetilde{U_n}=
{\langle \operatorname{tor}(L_n)\rangle}$
(\textit{cf.} 
\hyperref[t:trees1]{Theorem~\ref{t:trees1}(d)}).

Then by
\hyperref[p:2gen]{Proposition~\ref{p:2gen}(c)},
each group $L_n$ has a non-fictitious decomposition
$L_n=K_n\amalg_{D_n} W_n$
where $K_n$ are finite cyclic groups.
In light of 
\hyperref[p:treeprod]{Proposition~\ref{p:treeprod}},
and following its notation,
we have inverse systems
$\{K_n',\varphi_{nm}\}$ and $\{D_n'',\varphi_{nm}\}$
of groups conjugate to $K_n$ and $D_n$. 
Consider the two procyclic groups
$K\!:=\varprojlim K_n'$ and $D \!:=\varprojlim D_n''$.

We claim that $D=\{1\}$.
Note that since each $D_n$ is an edge stabilizer
with respect to the $L_n$-action, we have
$D=L\cap C^g$ for some $g\in G$.
Suppose on the contrary that $D\neq \{1\}$. 
Condition \textbf{C1$'$} says that
$C^g$ is a direct factor of $C_G(D)$,
hence $D$ is a direct factor of $C_L(D)$,
because $C_L(D) = L\cap C_G(D)$.
Since the procyclic group $K$ contains $D$,
it follows that $D=K$. 
Now, the projection $K\to K_{n_0}'$ is
surjective for some sufficiently large $n_0$, by
\hyperref[l:invsys_d-gen]{Lemma~\ref{l:invsys_d-gen}}.
Hence $D_{n_0}''=K_{n_0}'$; a contradiction to the
non-fictitious decomposition of $L_{n_0}$.
The claim is proved.

By
\hyperref[l:invsys_ncl]{Lemma~\ref{l:invsys_ncl}},
we have $\varprojlim {D_n''}^{L_n} =\{1\}$;
hence $L\cong \varprojlim L_{n}/{D_n''}^{L_n}$.
Now, if each $L_n/D_n''^{L_n}$ is procyclic, then
$L$ is procyclic.
So, we assume that each
$L_{n}/{D_n''}^{L_n}$ is $2$-generated.
Then, writing $L_n= K_n'\amalg_{D_n''} W_n'$
we have
$L\cong \varprojlim (K_n'/D_n'' \amalg W_n'/{D_n''}^{W_n'})$.
Since $K_n'/D_n''$ is $1$-generated,
so is $W_n'/{D_n''}^{W_n'}$.
Therefore $L\cong \mathbb{Z}_p\amalg \mathbb{Z}_p\,$,
by
\hyperref[p:decompfp]{Proposition~\ref{p:decompfp}}.
Our proof is finished for {\it Case 1}.

%\newpage
\medskip
\noindent Case 2.
\textit{$L/\widetilde{L}\cong {\mathbb Z}_p$.}
\medskip

For each $n$ we have 
$L_n/(\widetilde{L}/\widetilde{U_n}) \cong
L/\widetilde{L}\cong \mathbb{Z}_p$
%${\mathbb Z}_p\cong L/\widetilde L\cong L_n/(\widetilde L/\widetilde U_n)$
and therefore $\Gamma_n$ cannot be a tree.
Then we select a suitable edge $e_n$ of $\Gamma_n$,
set $\Delta_n:=\Gamma_n\,-\,\{e_n\}$, and present
$L_n=\operatorname{HNN}(K_n,D_n,t_n)$
where $D_n$ is the finite cyclic edge group of $e_n$
and $K_n$ is the fundamental pro-$p$ group
of graph of groups restricted to $\Delta_n$.
%with cyclic edge group $D_n$ of $e_n$ and
%$K_n=\Pi_1({\mathcal{G}_n}_{|\Delta_n},\Delta_n)$.

Since $\widetilde{L}/\widetilde{U_n}$ is generated by torsion,
as a consequence of
\hyperref[t:hnnproperties]{Theorem \ref{t:hnnproperties}(a)},
it follows that
$\widetilde{L}/\widetilde{U_n}$ is contained in ${K_n}^{L_n}$;
so, ${\langle \operatorname{tor}(L_n) \rangle}={K_n}^{L_n}$.
By \cite[Prop.~1.7(ii)]{Zalesskii:04},
%Let G be any virtually free pro-p group and 
%N a normal subgroup of G generated by torsion elements.
%Then G/<tor(G)> is free pro-p.
$K_n/{\langle \operatorname{tor}(K_n)\rangle}$
is a free pro-$p$ group,
whence ${\langle \operatorname{tor}(L_n) \rangle}$
has trivial image in the quotient
$\operatorname{HNN}
(K_n/{\langle \operatorname{tor}(K_n)\rangle},\{1\}, t_n)$
of $L_n$.
Thus $K_n={\langle \operatorname{tor}(K_n)\rangle}$.
Since $K_n$ acts on the pro-$p$ tree $\widetilde{U_n}\backslash T$
we have $K_n=\widetilde{K_n}$
(\textit{cf.} 
\hyperref[t:trees1]{Theorem~\ref{t:trees1}(d))},
so in particular, $\Delta_n$ must be a tree.

Passing now to a cofinal subset of $\mathbb N$,
if necessary, we may and do assume that for all $n$ either
$\Delta_n$ is a single vertex or $\Delta_n$ contains an edge.
We discuss the two subcases.

%\newpage
\medskip
\noindent Subcase 2($\alpha$).
\textit{For each $n$, the tree $\Delta_n$ is a single vertex.}
\medskip

Each $K_n$ is a finite $p$-group.
In the analysis of this subcase we make heavy use of
\hyperref[p:decomphnn]{Proposition~\ref{p:decomphnn}}
and its notation.

By
\hyperref[p:decomphnn]{Proposition~\ref{p:decomphnn}(a)},
we have an inverse system of
conjugates $K_n'$ of subgroups of $K_n$.
Passing again to a cofinal subset of $\mathbb N$, if necessary,
and making use of
\hyperref[p:2gen]{Proposition~\ref{p:2gen}(b)}
we may and do assume for each $n$ that,
either $K_n'$ is cyclic
or $d(K_n')=2$.
Thus, either
$K\!:=\varprojlim K_n'$ is procyclic
or $d(K)=2$.

If $K$ is procyclic, then for every $m$ there exists $n>m$
such that $\varphi_{nm}(K_n')$ is cyclic and so
$\varphi_{nm}(D_n')=\varphi_{nm}(D_n'^{t_n'})$,
where the $D_n'$ (resp. $t_n'$) are conjugates of 
$D_n$ and (resp. $t_n$)
according to 
\hyperref[p:decomphnn]{Proposition~\ref{p:decomphnn}}.
Hence $\varphi_{nm}(t_n')$ normalizes $\varphi_{nm}(D_n')$
and so $L_m=N_{L_m}(\varphi_{nm}(D_n'))$.
Since $L=\varprojlim L_m$ it follows that
$D\!:=\varprojlim D_m'$ is normal in $L$.
Recalling that  $E(T)$ is the compact set of the standard graph $T$  on which $L$ acts, setting in
\hyperref[l:fp]{Lemma~\ref{l:fp}}
$\Omega\!:=E(T)$ and $S_n\!:=D_n'$
we deduce the existence of an edge  $e\in E(T)$ with $D\subseteq G_e$.
Therefore $D^g\subseteq C$ for some $g\in G$. 
If $D\neq \{1\}$, making use of
\hyperref[l:cent]{Lemma~\ref{l:cent}},
we find that
%$L\cong {\mathbb Z}_p\times {\mathbb Z}_p$
$L$ is a free abelian pro-$p$ group
by %hypothesis~(i), 
condition \textbf{C1$'$},
as needed.

%Next assume that $D=1$.
If, on the other hand, $D=\{1\}$,
it follows from
\hyperref[l:invsys_ncl]{Lemma~\ref{l:invsys_ncl}}
that $\varprojlim D_m^{L_m}=1$ and
so $L=\varprojlim L_m/D_m^{L_m}$.
Observing that
$L_m/D_m^{L_m}
=K_m/(K_m\cap D_m^{L_m})\amalg \langle t_m\rangle$,
\hyperref[p:decompfp]{Proposition~\ref{p:decompfp}}
implies that
%$L\cong {\mathbb Z}_p\amalg {\mathbb Z}_p$,
$L$ is a free pro-$p$ group
whence the result, since $K$ is procyclic.

For finishing Subcase 2($\alpha)$
we assume that $d(K)=2$,
thus $d(K_n')=2$ for each $n$.
Then, setting $d:=2$ in
\hyperref[l:invsys_d-gen]{Lemma~\ref{l:invsys_d-gen}},
we are under the conditions of
\hyperref[p:decomphnn]{Proposition~\ref{p:decomphnn}(b)}.
Therefore,
we consider the inverse limit
of the inverse system of conjugates $D_n''$ or $D_n''^{t_n''}$
of the finite cyclic groups $D_n$.
Since $d(L)\le 2$,
by the argument in the previous paragraph,
we may and do assume that such inverse limit is non-trivial.
Then, passing to a cofinal subset, setting $d=1$ in
\hyperref[l:invsys_d-gen]{Lemma~\ref{l:invsys_d-gen}}
allows us to apply
\hyperref[p:decomphnn]{Proposition~\ref{p:decomphnn}(c)},
and obtain that $L= \operatorname{HNN}(K,D,t)$,
where $D\!:=\varprojlim D_n''$.

Now, by
\hyperref[l:fp]{Lemma~\ref{l:fp}},
$K$ stabilizes a vertex of $T$; 
so $K$ is, up to conjugation, contained in either $A$ or $B$ and
therefore it is a free pro-$p$ group or a free abelian pro-$p$ group,
by hypothesis~\textbf{C2}.
Note that,
since $E(T)$ is a compact space on which $L$ acts, setting in
\hyperref[l:fp]{Lemma~\ref{l:fp}}
$\Omega\!:=E(T)$ and $S_n\!:=D_n$,
we find $e\in E(T)$ with $D\subseteq G_e$.
Hence $D^g\subseteq C$ for suitable $g\in G$.
By condition
\textbf{C1$'$},
$C_L(D)$ is abelian
and we apply 
\hyperref[l:2gen_hnn]{Lemma~\ref{l:2gen_hnn}}
to settle Subcase 2($\alpha$).

%\newpage
\medskip
\noindent Subcase 2($\beta$).
\textit{For each $n$, the tree $\Delta_n$ contains an edge.}
\medskip

%\newpage
By
\hyperref[p:2gen]{Proposition~\ref{p:2gen}(c)},
each group $K_n$ has a non-fictitious decomposition
$K_n=X_n\amalg_{Z_n}W_n$ where
$X_n$ are finite cyclic groups.
Moreover, proceeding as in the proof of
\hyperref[p:treeprod]{Proposition~\ref{p:treeprod}}
(but using here that $\varphi_{nm}(L_n)=L_m$),
there exists an inverse system $\{X_n',\varphi_{nm}\}$ 
of conjugates of $X_n$ in $K_n$;
and we consider the procyclic subgroup $X=\varprojlim X_n'$.
We have two alternatives: $X$ is trivial, or not.

Suppose first that $X$ is trivial. %$X=1$.
Then, from
\hyperref[l:invsys_ncl]{Lemma~\ref{l:invsys_ncl}},
it follows that
$\varprojlim X_m'^{L_m}=\{1\}$ and
so $L\cong\varprojlim L_m/X_m'^{L_m}$.
Now, we have that
$L_m/X_m'^{L_m}$ is an $\operatorname{HNN}$-extension
with %cyclic associated subgroups and 
base group a finite cyclic quotient of $W_m$.
To see this observe that $K_m/X_m^{K_m}$ is finite cyclic
(since $K_m=\langle \operatorname{tor}(K_m)\rangle$
and $d(K_m)=2$) and so
$L_m/X_m'^{L_m}$ is an $\operatorname{HNN}$-extension
$\operatorname{HNN}(\bar K_m,\bar D_m,t_m)$ 
of a finite cyclic quotient $\bar K_m$ of $K_m$ where
the images of $D_m$ and $D_m^{t_m}$ coincide.
%Suppose that {H} acts on a simply connected graph {Sigma}, ​​so that the quotient {Gamma = H \ Sigma} is finite. Then {H} is the fundamental group of a graph of groups {(\ mathcal {G}, Gamma)}, where for each {m} in {Gamma}, {\mathcal{G} (m)} is isomorphic to the stabilizer in {H} of some pre-image of {m} in {Sigma}
Hence, 
from Subcase 2($\alpha$),
$L$ is a free pro-$p$ group or a free abelian pro-$p$ group.

Finally, suppose $X$ is non-trivial.
By
\hyperref[l:invsys_d-gen]{Lemma~\ref{l:invsys_d-gen}}
with $d=1$, we obtain that $\varphi_{nm}(X_n')=X_m'$.
So, by
\hyperref[p:treeprod]{Proposition~\ref{p:treeprod}},
there exists an inverse system $\{Z_n'',\varphi_{nm}\}$ 
of conjugates of $Z_n$ in $X_n'$;
and we consider the procyclic group $Z=\varprojlim Z_n''$.
We have $Z\neq X$, otherwise by
\hyperref[l:invsys_d-gen]{Lemma~\ref{l:invsys_d-gen}}
we could find $n$ with $Z_n''=X_n'$;
contradicting the non-fictitious decomposition
$K_n=X_n\amalg_{Z_n}W_n$.
Setting 
$\Omega\!:=E(T)$ and $S_n\!:=Z_n''$ in
\hyperref[l:fp]{Lemma~\ref{l:fp}},
we obtain $e\in E(T)$ with $Z\subseteq L_e$.
Hence there exists $g\in G$ with $Z^g\subseteq C$.
Now, since $Z\neq X$, 
condition
\textbf{C1$'$}
implies $Z=\{1\}$.
%explain?
Therefore
$L\cong \varprojlim L_n/Z_n^{L_n}$.

Now, the group $L_n/Z_n^{L_n}$ can be seen as 
the quotient group
$K_n/{Z_n}^{K_n}\amalg \langle t_n\rangle$ modulo
a single relation
which comes from the relation between the associated subgroups
of the $\operatorname{HNN}$-extension.
Precisely,
let us denote by ``bar'' %the image of a subgroup of $K_n$ in
passing to the quotient $K_n/{Z_n}^{K_n}$, and
write
$\overline{K_n}=A_n\amalg B_n$
with
$A_n\cong X_n/{Z_n}^{X_n}=X_n/{Z_n}$
and $B_n\cong W_n/{Z_n}^{W_n}$.
The quotient we analyze is
$(\overline{K_n}\amalg \langle t_n\rangle)
/\langle t_n^{-1}\overline{d_{0n}}t_n\overline{d_{1n}}^{-1}\rangle
^{\overline{K_n}\amalg \langle t_n\rangle}$,
where $d_{0n}$ is a generator of ${D_n}$ and
$d_{1n}$ is the corresponding generator of ${{D_n}^{\!\!t_n}}$.

By
\hyperref[t:afpproperties]{Theorem~\ref{t:afpproperties}(a)},
the finite cyclic groups
$\overline{D_n}$ and $\overline{{D_n}^{\!\!t_n}}$
are, up to conjugation by an element of $\overline{K_n}$,
contained in $A_n$ or $B_n$.

%\newpage
Suppose that,
up to conjugation, both
$\overline{D_n}$ and $\overline{{D_n}^{\!\!t_n}}$
do not coincide with the free factors containing them.
Then both are contained in the Frattini $\Phi(\overline{K_n})$
and hence
$\langle t_n^{-1}\overline{d_{0n}}t_n\overline{d_{1n}}^{-1}\rangle
^{\overline{K_n}\amalg \langle t_n\rangle}$
is contained in $\Phi(\overline{K_n}\amalg \langle t_n\rangle)$.
So, 
$\overline{K_n}/\Phi(\overline{K_n}) \amalg 
\langle t_n\rangle/\Phi(\langle t_n\rangle)$
is a quotient of the group
$(\overline{K_n}\amalg \langle t_n\rangle)
/\langle t_n^{-1}\overline{d_{0n}}t_n\overline{d_{1n}}^{-1}\rangle
^{\overline{K_n}\amalg \langle t_n\rangle}$.
Since $d(L_n)\le 2$, we must have $d(\overline{K_n})=1$
and $B_n$ is trivial; hence $W_n=Z_n$.
This is a contradiction to the non-fictitious decomposition of
$K_n=X_n\amalg_{Z_n}W_n$.

Otherwise, without loss of generality,
we suppose that $\overline{D_n}$ coincides, 
up to conjugation, with $A_n$.
Changing $A_n$ by a conjugate, if necessary,
we have $\overline{D_n}=A_n$.
On the other hand, there exists an element $y_n$ in $\overline{K_n}$
such that $\overline{d_{1n}}=w_n^{y_n^{-1}}$
for some $w_n$ belonging either to $\overline{D_n}$ or to $B_n$.
Letting $z_n:=t_ny_n$ we have
\[
\begin{array}{ll} \!\!
(\overline{K_n}\amalg \langle t_n\rangle)
/\langle t_n^{-1}\overline{d_{0n}}t_n\overline{d_{1n}}^{-1}\rangle
^{\overline{K_n}\amalg \langle t_n\rangle} \!\!\!\!
&= (\overline{D_n}\amalg B_n \amalg \langle t_n\rangle)
/\langle \overline{d_{0n}}(t_n\overline{d_{1n}}^{-1}t_n^{-1})\rangle
^{\small\overline{K_n}\amalg \langle t_n\rangle}\\
&\cong (\overline{D_n}\amalg B_n \amalg \langle z_n\rangle)
/\langle \overline{\small d_{0n}}z_nw_n^{-1}z_n^{-1}\rangle
^{\overline{K_n}\amalg \langle z_n\rangle}.\end{array}
\]

If $w_n\in B_n$, 
then the desired quotient
is isomorphic to the free pro-$p$ product of
a quotient of $B_n$ and $\langle z_n\rangle$,
by eliminating the generator $\overline{d_{0n}}$ of $\overline{D_n}$.
Since $d(L_n)\le 2$, it follows from
\hyperref[p:decompfp]{Proposition~\ref{p:decompfp}(a)}
that $L$ is a free pro-$p$ group.

Suppose now that $w_n$ is a power of $\overline{d_{0n}}$.
Then $z_n$ normalizes the finite group $A_n$ in 
$(\overline{D_n}\amalg B_n \amalg \langle z_n\rangle)
/\langle \overline{d_{0n}}z_nw_n^{-1}z_n^{-1}\rangle
^{\overline{K_n}\amalg \langle z_n\rangle}$.
Let ``prime'' denote passing to Frattini quotients,
we consider the following image of the desired quotient
${(\overline{K_n}'\times \langle z_n\rangle')}/ \\ %%%linebreaking!!!
{\langle z_n'^{-1}\overline{d_{0n}}'z_n' w_n'^{-1}\rangle
^{\overline{K_n}'\times \langle z_n\rangle}}$.
Note that $w_n'=\overline{d_{0n}}'$ in this image.
Indeed, since $z_n$ acts by conjugation on the finite group $A_n$
we have that $z_n$ acts trivially
on the group $A_n'=A_n/\Phi(A_n)$.
Thus the considered image is simply
$\overline{K_n}'\times \langle z_n\rangle'$.
Since $d(L_n)\le 2$, $B_n$ must be trivial.
This contradicts the non-fictitious decomposition of
$K_n=X_n\amalg_{Z_n}W_n$.

The proof of the theorem is concluded.
\end{proof}

%\newpage
\begin{cor}\label{c:2-free}
Suppose that 
%%neither $A$ nor $B$ contains a
%%$2$-generated non-procyclic abelian subgroup.
%every $2$-generated abelian subgroup of $A$ and 
%every $2$-generated abelian subgroup of $B$ is procyclic.
every $2$-generated subgroup of $A$ and 
every $2$-generated subgroup of $B$ is 
a free pro-$p$ group.
Then every $2$-generated subgroup %$L$ 
of $G$
is also a free pro-$p$ group.
%Then so is any $2$-generated pro-$p$ subgroup of $G$.
\end{cor}

\begin{proof}
Suppose that $L$ is a free abelian pro-$p$ group of rank $2$
contained in $G$.
Let $T$ be the standard pro-$p$ tree on which $G$ acts.

In virtue of
\cite[Thm.~3.18]{RZ:00a}
either $L$ stabilizes a vertex or there
is an edge $e$ of $T$ such that $L/L_e\cong{\mathbb Z}_p$.
But $L$ cannot stabilize a vertex,
else it would be conjugate to 
a subgroup of one of the amalgamated free factors of $G$,
contradicting the hypothesis.

Therefore $L/L_e\cong {\mathbb Z}_p$ for some edge $e$ of $T$.
Since $d(L)=2$ we must have $L_e \not= \{1\}$.
Conjugating $L$ by some element of $G$
we may and do assume that $L_e$ is contained in $C$.
Now, we have $N_G(L_e)={\langle C_{A}(C), C_{B}(C)\rangle}$,
from
%the last paragraph of the proof of
\hyperref[l:cent]{Lemma~\ref{l:cent}};
and the hypothesis of the corollary together with 
condition \textbf{C1}
imply $C_A(C)=C=C_B(C)$.
Therefore $L=N_L(L_e)\cong \mathbb{Z}_p$; another contradiction.

So, by
\hyperref[t:freeorabelian_intro]{Theorem~\ref{t:freeorabelian_intro}},
all $2$-generated subgroups of $G$ must be free pro-$p$.
\end{proof}

%\newpage
As mentioned in the
\hyperref[s:intro]{Introduction}, our
\hyperref[t:freeorabelian_intro]{Theorem~\ref{t:freeorabelian_intro}}
is a pro-$p$ version of
\cite[Thm.~2]{Baumslag:62}
and, of 
\cite[p.~601]{BBaumslag:68}
for free products with cyclic amalgamations 
whose amalgamating subgroups are malnormal in both factors; 
it also generalizes \cite[Thm.~7.3]{KZ:11}.

Note that
although our last results do not deal with trivial amalgamations,
in virtue of pro-$p$ versions of the
Kurosh subgroup theorem
and the
Grushko-Neumann theorem
(\textit{e.g.}, 
\cite[Thm.~4.3]{Melnikov:90}  
and \cite[Thm.~9.1.15]{RZ:10}),
the corresponding results of
\hyperref[t:freeorabelian_intro]{Theorem~\ref{t:freeorabelian_intro}}
and
\hyperref[c:2-free]{Corollary~\ref{c:2-free}}
also hold for free pro-$p$ products.

%\newpage
We end this section with a simple example
not covered by previous results in the literature.

\begin{example} \label{ex:demushkin}
\upshape
Let $D$ be a non-soluble Demushkin group
(\textit{e.g.}, the pro-$p$ completion of a surface group of genus $\ge 2$),
and let $F$ be a non-abelian free pro-$p$ group of finite rank.
If $G=D\amalg_C F$, where
$C$ is a maximal procyclic subgroup in $D$ and $F$,
then by
\hyperref[c:2-free]{Corollary~\ref{c:2-free}}
any $2$-generated subgroup of $G$ is a free pro-$p$ group.
In fact,
any $2$-generated subgroup of $D$ is a free pro-$p$ group,
and $N_G(C)=N_D(C)\amalg_C N_F(C)=C\amalg_C C=C$
(\textit{cf.} \cite[Ex.~5(b) and Ex.~6, p.~41]{Serre:94}). %p. 44 em 1997
\end{example}

%%%%%%%%%%%%%%%%%%%%%%%%%%%%%%
%\newpage
\phantomsection
\section*{Acknowledgements}
\label{s:acknowl}
%\addcontentsline{toc}{section}{Acknowledgements}
%%%%%%%%%%%%%%%%%%%%%%%%%%%%%%

The second and third named authors are grateful for the
partial financial support from CNPq and CAPES.

%%%%%%%%%%%%%%%%%%%%%%%%%%%%%%
%\newpage

\end{document}